\documentclass[reqno,12pt]{amsart}

\usepackage{graphicx}
\usepackage[pdfborder={0 0 0}]{hyperref}
\usepackage{lipsum}
\usepackage{amsmath}
\usepackage{amssymb}
\usepackage{amsthm}
\usepackage{amsrefs}
\usepackage[usenames,dvipsnames]{color}
\usepackage{float}
\usepackage{geometry}
\DeclareMathOperator{\E}{\mathbb E}

\usepackage[T1]{fontenc}
\usepackage{amsmath, amssymb, amsthm,amsfonts} 
\usepackage{float}
\usepackage{tikz}
\usetikzlibrary{arrows,calc,shapes,decorations.pathreplacing}
\usepackage{hyperref}
\hypersetup{pdftex,colorlinks=true,allcolors=blue}
\usepackage{caption}
\usepackage{pgfplots}
\usepackage{subcaption}
\usepackage{enumitem}
\usepackage{mathrsfs} 

\newenvironment{customlegend}[1][]{%
    \begingroup
    \csname pgfplots@init@cleared@structures\endcsname
    \pgfplotsset{#1}%
}{%
    \csname pgfplots@createlegend\endcsname
    \endgroup
}%
\def\addlegendimage{\csname pgfplots@addlegendimage\endcsname}

\newtheorem{theorem}{Theorem}
\newtheorem{lemma}[theorem]{Lemma}
\newtheorem{proposition}[theorem]{Proposition}
\newtheorem{corollary}[theorem]{Corollary}

\newtheorem*{theorem*}{Theorem}

\theoremstyle{remark}
\newtheorem{example}[theorem]{Example}
\newtheorem{remark}{Remark}
\newtheorem*{remark*}{Remark}

\def\P{\mathrm{P}}
\def\E{\mathrm{E}}
\def\Var{\mathrm{Var}}
\def\Cov{\mathrm{Cov}}

\newcommand{\bs}{\boldsymbol}

\newcommand{\vb}{\vspace{3.2mm}}

\newcommand{\apost }{'}

\newcommand{\darrow}{\:\to_{\rm d}\:}
\newcommand{\asarrow}{\:\to_{\rm as}\:}

\newcommand*\diff{\mathop{}\!\mathrm{d}}
\newcommand{\floor}[1]{\left\lfloor #1 \right\rfloor}

\title[Hypothesis testing for a L\'evy-driven storage system]{Hypothesis testing for a L\'evy-driven \\storage system by Poisson sampling\footnote{To appear in Stochastic Processes and their Applications}}

\author[M. Mandjes and L. Ravner]{M. Mandjes and L. Ravner}

\date{\today}

\begin{document}

\begin{abstract} This paper focuses on hypothesis testing for the input of a L\'evy-driven storage system by sampling of the storage level. As the likelihood is not explicit we propose two tests that rely on transformation of the data. The first approach uses i.i.d. `quasi-busy-periods' between observations of zero workload. The distribution of the duration of quasi-busy-periods is determined. The second method is a conditional likelihood ratio test based on the Bernoulli events of observing a zero or positive workload, conditional on the previous workload. Performance analysis is presented for both tests along with speed-of-convergence results, that are of independent interest.
\vspace{3.2mm}

\noindent
{\sc Keywords.} L\'evy-driven storage system $\circ$ Poisson sampling $\circ$ hypothesis testing $\circ$ convergence to stationarity

\vb

\noindent
{\sc Affiliations.} 
{\it Michel Mandjes} is with Korteweg-de Vries Institute for Mathematics, University of Amsterdam, Science Park 904, 1098 XH Amsterdam, the Netherlands. He is also with E{\sc urandom}, Eindhoven University of Technology, Eindhoven, the Netherlands, and Amsterdam Business School, Faculty of Economics and Business, University of Amsterdam, Amsterdam, the Netherlands. {\it Liron Ravner} is with the Department of Statistics in the University of Haifa, Israel, as well as Korteweg-de Vries Institute for Mathematics, University of Amsterdam, Science Park 904, 1098 XH Amsterdam, the Netherlands. Both authors' research is partly funded by NWO Gravitation project N{\sc etworks}, grant number 024.002.003.

\vb

\noindent
{\sc Acknowledgments.} The authors would like to thank an anonymous referee for his/her useful feedback and comments. The authors are also grateful to O.\ Kella (Hebrew University, Jerusalem, Israel) for his comments and for the fruitful discussions during the writing of this paper. 

\vb

\end{abstract}

\maketitle

\newpage

\maketitle
%

\section{Introduction}\label{sec:intro}
In statistical hypothesis testing, one wishes to distinguish between a null hypothesis H$_0$ and an alternative hypothesis H$_1$ by observing a series of random variables. In the common setup the hypotheses directly relate to the observations: for instance, observing a series of random variables, the hypotheses could correspond to these observations stemming from two specific distributions. In the statistical literature a vast body of results has been established that facilitate such tests. In many applications, however, the hypotheses relate to the observations in a more involved manner. A prominent example of such a situation lies in the domain of storage systems and queues: one has periodic observations of the storage level, but the hypotheses are in terms of the system's input process. For instance, by periodically observing the workload, one would like to distinguish between two values of the arrival rate. While some work on hypothesis testing for this context has been done, a general framework is still lacking, and, as a consequence, various open questions remain.

In this paper we consider the workload process of a rather broad set of storage systems. We focus on the situation of a resource that is fed by an increasing L\'evy process (often referred to as a `subordinator') which is uniquely characterized by its Laplace exponent $\varphi(\cdot)$, and that is emptied at a deterministic rate. This class of workload models covers the intensively studied storage system with compound Poisson input (often referred to as the M/G/1 queue), but it allows the driving L\'evy process to be any subordinator (for instance a Gamma process or an inverse Gaussian process).

The main objective of this paper is to develop methods for distinguishing between two characteristic exponent functions, say $\varphi_0(\cdot)$ and $\varphi_1(\cdot)$, based on observations of the corresponding workload process, rather than on observations of the L\'evy input processes themselves. It is assumed that the workload is observed at Poisson instants. A major complication is that, although in our L\'evy-input context the increments of the cumulative input process are independent, subsequent workload observations are {\it not}, so that the likelihood cannot be evaluated and consequently conventional tests cannot be applied. 

\vb

\noindent{\em Contributions.} 
As mentioned, this paper develops tests for distinguishing between Laplace exponents $\varphi_0(\cdot)$ and $\varphi_1(\cdot)$ based on workload observations. We present two approaches that succeed in resolving the complications identified above. A key feature of both approaches is that the sequential test has power one; if the null hypothesis is wrong, then this will be detected with probability one as long as there is no restriction on the number of observations. However, a type-I error is still possible, i.e., a false rejection of the null hypothesis. To assess the performance of the test, one wishes to compute (or approximate) the type-I error probability.\begin{itemize}
\item[$\circ$]
In the first approach we consider the so-called `quasi busy period', which is defined as the number of Poisson observations until the workload hits 0 again. It is an inherent feature of the model that subsequent quasi busy periods are i.i.d. As a consequence, a traditional likelihood ratio test can be used. To evaluate the likelihood, the distribution of the duration of quasi busy periods should be determined; we point out how this can be done. In addition, we present results that assess the performance of the test. 
\item[$\circ$]The second approach works with a conditional likelihood ratio test. It is based on the fact that we can explicitly compute the probability of the events of observing a zero or positive workload, conditional on the value of the previous observation. In this setup the increments of the log-likelihood are not i.i.d., so that a rather delicate analysis is needed to analyze the performance of the test.
\end{itemize}
To quantify the performance of the test pertaining to the second approach, we require a number of results describing the convergence of a L\'evy-driven storage system to its stationary version. To this end, we derive a series of speed-of-convergence results, that are also of independent interest. 

\vb

\noindent{\em Background and related literature.} 
In the situation we are considering the systems's input is a non-decreasing L\'evy process minus a deterministic drift, while the workload is sampled according at Poisson epochs. This setup has been considered before in \cite{RBM2019}, where a method was developed for consistent and asymptotically normal semi-parametric estimation of the Laplace exponent, based on workload observations at Poisson epochs. The present paper can be seen as the hypothesis testing counterpart of \cite{RBM2019}.

The two approaches we pursue are attempts to remain as closely as possible to conventional likelihood ratio tests, both in terms of the evaluation of the likelihood and the assessment of the test's performance. However, as pointed out above, due to workload's intricate dynamics, various complications need to be overcome. For more general background on hypothesis testing we refer to the textbooks \cite{book_S2013,book_Y2013}. Our methods naturally extend to a stability test for a queue with unknown input as well as changepoint detection.

As mentioned above, in our storage system setting a general theory for hypothesis testing is lacking. A review of the classical literature on this topic is given in \cite[Section 5]{BR1987}. Without attempting to provide an exhaustive overview, we mention a few specific references. If Markov chains can be embedded in the queueing process (which is the case in e.g.\ systems of the type M/G/1 and GI/M/$s$), and if the corresponding jump process is observed, the likelihood can be evaluated in closed-form \cite[Section 5.e]{BR1987}; cf.\ the sequential test proposed in \cite{BR1972}. A large sample asymptotic test for the traffic intensity of a G/G/$s$ system in which interarrival and service times are observed, is presented in \cite{RH1986}; relying on a delta-method type argument the authors construct a normal approximation for the error probabilities of the test. A test based on the distribution of the number of arrivals during a service period is proposed in \cite{RBH1984}. 

\vb

\noindent{\em Organization.} This paper is organized as follows. Section \ref{sec:model} describes the setup considered in this paper, presents some preliminaries on L\'evy-driven storage systems and hypothesis testing, and formally states our objectives. Then in Section \ref{sec:QBP} we detail the quasi-busy-period based approach, including the analysis of its performance. Then, in Section \ref{sec:converge} we provide a series of results on convergence of the storage-level process to its stationary version, which are applied in Section \ref{sec:cond_LRT} but which are relevant in their own right as well. Then Section \ref{sec:cond_LRT} presents the conditional likelihood ratio approach and its performance analysis. Numerical experiments are described in Section \ref{sec:simulation}, while Section \ref{sec:DCR} provides a brief discussion and some concluding remarks.

\section{Model, preliminaries, objectives}\label{sec:model}
In this section we provide a model description (and introduce the notation that will be used throughout the paper), present some preliminaries on L\'evy-driven storage systems, and state our objectives.

\subsection{Model} We consider a storage system fed by a non-decreasing L\'evy process $(J(t))_{t\geqslant 0}\equiv J(\cdot)$. The output of the system is a unit-rate linear drift. The system's {\it net} input process $(X(t))_{t\geqslant 0}\equiv X(\cdot)$ is therefore given through $X(t)=J(t)-t$, which is a spectrally-positive L\'evy process characterized by its Laplace exponent
\begin{equation*}
\varphi(\alpha):=\log\E  {\rm e}^{-\alpha X(1)}=\alpha-\int_{(0,\infty)}(1- {\rm e}^{-\alpha x})\nu(\diff x)\ ,
\end{equation*}
where $\nu$ is a L\'evy jump measure such that $\nu(-\infty,0)=0$. 

In the sequel we let $(V(t))_{t\geqslant 0}\equiv V(\cdot)$ denote the corresponding workload process. It can be represented as the net input process reflected at zero, in that $V(t)=X(t)+\max\{V(0),L(t)\}$, where $L(t):=-\inf_{0\leqslant s\leqslant t}X(s)$. 
Under the stability condition $\varphi'(0)=-\E X(1)>0$ the workload has a stationary distribution $V:=V(\infty)$ with an LST given by the {\it generalized Pollaczek-Khintchine} formula \cite[p.\ 27]{book_DM2015}:
\begin{equation}\label{eq:GPK}
\E  {\rm e}^{-\alpha V}=\frac{\alpha \varphi\apost (0)}{\varphi(\alpha)}\ ;
\end{equation}
otherwise the workload process is unstable, meaning that $V(\infty)=\infty$ almost surely. The first two moments of the stationary distribution are given by
\begin{equation}\label{eq:V_moments}
\E V = \frac{\varphi^{(2)}(0)}{2\varphi\apost(0)}\ , \ \E V^2 = \frac{1}{2}\left( \frac{\varphi^{(2)}(0)}{\varphi\apost(0)}\right)^2- \frac{\varphi^{(3)}(0)}{3\varphi\apost(0)} \ ,
\end{equation}
where
\[
\varphi^{(k)}(0)=\lim_{\alpha \downarrow 0}\frac{{\rm d}^k}{{\rm d}\alpha^k}\varphi(\alpha)=(-1)^k\int_{(0,\infty)}x^k\nu( \diff x)\ ,k\geqslant 2 \ .
\]
Let $\rho:=\E J(1)=\int_{(0,\infty)}x\nu(\diff x)$ denote the expected input per unit of time. Then the expected net input per unit of time is
\[
\E X(1)=-\varphi\apost(0)=-1+\int_{(0,\infty)}x\nu( \diff x)=\rho-1\ .
\]
Therefore, $\E X(1)<0$ is equivalent to $\rho<1$, which is the typical form of the stability condition in queueing theory. 

\vb

\begin{example} (M/G/1 system)
An important special case corresponds to the input process being compound Poisson, with arrival rate $\lambda$ and i.i.d.\ job sizes $B_1,B_2,\ldots$ that are distributed as a generic non-negative random variable $B$ with distribution function $G(\cdot)$. In this case we have that $\nu(\diff x)=\lambda G(\diff x)$, $\varphi(\alpha)=\lambda\left(G^\star(\alpha)-1\right)+\alpha$, where $G^\star(\alpha):=\E  {\rm e}^{-\alpha B}$, and $\rho=\lambda\,\E B$. As is well known, $\rho<1$ is a necessary and sufficient condition for the stability of the system. \hfill $\Diamond$
\end{example}

\subsection{Poisson sampling} In this paper the workload process will be sampled at Poisson epochs. Our analysis strongly relies on the availability of explicit expressions of the workload after an exponentially distributed time as a function of the initial workload level.

The workload process starts, at time $0$, at some known level $V(0)$. It is observed according to an independent Poisson process with rate $\xi>0$. Let $T_1,T_2,\ldots$ be i.i.d.\ exponentially distributed with parameter $\xi>0$. We denote by $S_k:=T_1+\ldots+ T_k$ the epoch of the $k$-th observation, i.e., $S_k$ has an Erlang distribution with scale parameter $k$ and and shape parameter $\xi$. From now on we use the compact notation $V_k:= V(S_k)$ to denote the workload process at the $k$-th Poisson epoch, and in addition $V_0:=V(0)$. 

If the input process is spectrally positive, then the distribution of the workload at sample $i\in \{1,2,\ldots\}$ conditional on the workload at sample $i-1$ is characterized through
\begin{equation}\label{eq:V_LST_transient}
\E\left[ {\rm e}^{-\alpha V_i}\,|\,V_{i-1}\right]=\frac{\xi}{\xi-\varphi(\alpha)}\left( {\rm e}^{-\alpha V_{i-1}}-\frac{\alpha}{\psi(\xi)} {\rm e}^{-\psi(\xi)V_{i-1}}\right)\ ,
\end{equation}
where $\psi(x):=\varphi^{-1}(x)$; see e.g.\ \cite{KBM2006} and \cite[Ch. IV]{book_DM2015}. If, furthermore, the input process of the queue is a subordinator, then the workload process attains the value zero with positive probability. This probability can be computed by taking $\alpha\to\infty$ in \eqref{eq:V_LST_transient}:
\begin{equation}\label{eq:P_idle_transient}
\P(V_i=0\,|\,V_{i-1}=v)=\frac{\xi}{\psi(\xi)} {\rm e}^{-\psi(\xi)v}\ .
\end{equation}
This probability will be the main building block of the conditional likelihood ratio test that will be introduced in Section \ref{sec:cond_LRT}. In Section \ref{sec:QBP} we will present new results on the distribution of the number of samples taken between consecutive observations of zero workload.

Note that when $\xi\downarrow 0$, if $\rho<1$ then $\psi(\xi)\to 0$ whereas otherwise $\psi(\xi)\to \psi(0)>0$. If $\rho<1$ then, for all $v\geqslant 0$,
\[
\frac{\xi}{\psi(\xi)} {\rm e}^{-\psi(\xi)v}\xrightarrow{\xi\downarrow 0} 1-\rho \ .
\]
This relation has an intuitive backing: when decreasing the sampling rate the events of finding an idle server become `increasingly independent', and therefore the probability approaches the steady-state  idle-server probability $1-\rho$.

\subsection{Hypothesis testing} 
Our goal is to provide a framework for testing hypotheses related to the distribution of the net input of the queue: with some abuse of notation,
\begin{equation}\label{T1}
\begin{array}{cc}
{\rm H}_0: & (X(t))_{t\geqslant 0}=_{\rm d} \varphi_0(\cdot) \ , \\
{\rm H}_1: & (X(t))_{t\geqslant 0}=_{\rm d} \varphi_1(\cdot)\ .
\end{array} 
\end{equation}
An important special case concerns tests that correspond to the traffic intensity:
\[
\begin{array}{cc}
{\rm H}_0: & \E X(1)=-\varphi_0\apost(0) \ , \\
{\rm H}_1: & \E X(1)=-\varphi_1\apost(0)\ ,
\end{array} 
\]
and in particular detecting stability if $\varphi_0\apost(0)>0$ and $\varphi_1\apost(0)\leqslant 0$. {Note that even in the case of stability detection the tests presented here require full characterization of the distributions corresponding to ${\rm H}_0$ and ${\rm H}_1$. In other words, the test can distinguish between specific input distributions (or a collection of such distributions) that yield a stable or unstable workload process. In Section \ref{sec:DCR} we describe a possible extension to a framework allowing for composite hypotheses, and such an extension can also be useful for the stability detection problem.} Another objective concerns providing a procedure for changepoint detection. In this context the system may start with Laplace exponent $\varphi_0(\cdot)$, but a change may occur at a random time after which the Laplace exponent becomes $\varphi_1(\cdot)$. The goal is to identify if (and when) this change occurs. 

In many applications one is interested in a dynamic test as observations are collected, as opposed to a static test with a given sample of size $n$. Then the test is defined by a test statistic $g_n(V_0,\ldots,V_n)$, in combination with two disjoint decision sets, $\mathscr{G}_0$ and $\mathscr{G}_1$. Let \[N_k:=\inf\{n: \ g_n(V_0,\ldots,V_n)\in \mathscr{G}_k\}\]for $k\in\{0,1\}$. The stopping rule is: collect samples until the stopping time $N:=\min\{N_0,N_1\}$, and reject (accept, respectively) the null hypothesis if $N_1<N_0$ ($N_0<N_1$, respectively). As usual, there are two types of possible errors: the type-I error is quantified as $\alpha:={\rm P}_{{\rm H}_0}(N_1<N_0)$, whereas the type-II error is $1-\pi:={\rm P}_{{\rm H}_1}(N_1>N_0)$, where $\pi$ is typically referred to as the power of the test. In some cases the sampling is stopped only if the null hypothesis is rejected, for example in the context of change-point detection. If $\P_{{\rm H}_1}(N<\infty)=1$, then the test is called a power-one sequential test. The significance level of the test is $\P_{{\rm H}_0}(N<\infty)$. 

The most common approach for hypothesis testing is the likelihood-ratio-test (LRT): reject ${\rm H}_0$ if $L_n\geqslant x$ for some $x>0$, where
\[
L_n:= \frac{\P_{{\rm H}_1}(V_{1},\ldots, V_n)}{\P_{{\rm H}_0}(V_{1},\ldots, V_n)}\ .
\]
This test is proven to have certain optimal properties, such as having the highest power for any value of $\alpha$. In our setting, however, serious complications arise, as we do not have closed forms expressions for the likelihoods $\P_{{\rm H}_i}(V_{1},\ldots, V_n)$, $i=0,1$. This is a consequence of the fact that the observations $V_i$ are in general not identically distributed (as we do not necessarily start at time $0$ with a stationary workload), and, more importantly, have a rather elaborate dependence structure. Therefore, if we would like to apply the above LRT, then a first idea would be to rely on numerical or approximate techniques to evaluate the likelihoods $\P_{{\rm H}_i}(V_{1},\ldots, V_n)$. However, the remedies that we propose, and which will be described in detail later, are of a different nature: we transform the observations $V_{1},\ldots, V_n$ into a vector of which we can compute the likelihood, effectively bringing us back into a classical LRT framework.

\begin{remark}\label{R1}
A (na\"{\i}ve) benchmark test on the mean rate generated by the driving L\'evy process is the following. Perform a simple threshold test on the average workload: reject ${\rm H}_0$ if $n^{-1}\sum_{i=1}^nV_i \geqslant x$ for a suitably chosen $x$. In case one is interested in detecting stability, then this is a power-one test as in the unstable setting the average workload will grow unbounded. However, if both hypotheses correspond to stable queues, then there is a positive probability of a type-II error (but this can be controlled by choosing the test's parameters appropriately). Large-sample asymptotics for the test statistic are readily available, such as those in e.g.\ \cite{GMW1993}: after centering and normalizing by $\sqrt{n}$ it converges to a normal random variable. This CLT could be used for an approximative likelihood-ratio test corresponding to i.i.d.\ normal random variables, but this would ignore the dependence between the observations. The aim of this paper is to develop new tests that combine the advantages of the likelihood-ratio test with tractable methods for quantifying its performance. We get back to the na\"{\i}ve test described above in the numerical analysis of Section \ref{sec:simulation}.\hfill$\Diamond$
\end{remark}

\section{Approach I: quasi busy periods}\label{sec:QBP}
As mentioned in the previous section, the approach we take is to transform the vector of workload observations (performed at Poisson instances) into an alternative vector for which LRT-type tests can be performed. In this section we focus on using the {\it quasi busy period} (QBP), being the number of observations until the workload hits zero again. Evidently, these QBPs constitute a sequence of i.i.d.\ random variables. In Section \ref{S31} we derive distributional properties of the QBP, which enable us to evaluate the likelihood. Section \ref{S32} then describes how the LRT should be set up.

\subsection{Distribution of quasi busy period}\label{S31}
In this section we concentrate on the random quantity, conditional on $V_0=0$,
\[
R:= \inf\{k\in{\mathbb N}: V_k = 0, V_0=0\}\ ;
\] 
$R$ thus records how many Poisson epochs it takes until an empty buffer is observed again. The objective of this subsection is to devise a procedure that facilitates the computation of the distribution of $R$, through the probabilities
\[
r_k(\xi) := \P(R=k\,|\,V_0=0)\ .
\]
We do so by first computing the probabilities {$p_k(\xi) :=\P(V_k=0|V_0=0)$,} after which we express the $r_k(\xi)$ in terms of the $p_k(\xi)$. 

Let $S_k$ be an Erlang random variable with scale parameter $n$ and shape parameter $\xi$, i.e., the sum of $n$ i.i.d.\ exponentially distributed random variables with mean $\xi^{-1}.$ We first point out how to compute, with $f_{S_k}(\cdot)$ denoting the density of $S_k$,
\begin{align*}
p_k(\xi) &={\P(V_k=0\,|\,V_0=0)}= \int_0^\infty f_{S_k}(t) \,\P(V(t)=0\,|\,V(0)=0)\diff t\\
&= \int_0^\infty \frac{\xi^k t^{k-1}}{(k-1)!} \,{\rm e}^{-\xi t} \,\P(V(t)=0\,|\,V(0)=0)\diff t\ .
\end{align*}
To this end, observe that
\begin{align*}
p_k(\xi) &=- \frac\diff {\diff \xi} \int_0^\infty \frac{\xi^k t^{k-2}}{(k-1)!} \,{\rm e}^{-\xi t}\, \P(V(t)=0\,|\,V(0)=0)\diff t\:+\\&\:\:\:\:\:\hspace{2.1cm} \int_0^\infty \frac{k\xi^{k-1} t^{k-2}}{(k-1)!} \,{\rm e}^{-\xi t}\, \P(V(t)=0\,|\,V(0)=0)\diff t\ .
\end{align*}
In other words, we obtain the recursion, for $k=2,3,\dots$,
\[
p_k(\xi) =-\frac\diff {\diff \xi}\left(\frac{\xi}{k-1}\,p_{k-1}(\xi)\right)+\frac{k}{k-1}\, p_{k-1}(\xi)
=p_{k-1}(\xi)-\frac{\xi}{k-1}\,p'_{k-1}(\xi)\ .
\]
The initialization of the recursion follows from $p_1(\xi) = \xi/\psi(\xi).$ The next term is
\[
p_2(\xi)= p_1(\xi)-\xi p_1'(\xi)= \frac{\xi}{\psi(\xi)} -\xi\frac{\psi(\xi) -\xi\psi'(\xi)}{\psi^2(\xi)}=\left(\frac{\xi}{\psi(\xi)}\right)^2\psi'(\xi)\ .
\]
This recursion can be further expanded, so as to obtain the following result. 
\begin{proposition} \label{P1} 
For $k\in{\mathbb N}$,
\begin{equation} \label{E1} 
p_k(\xi) =
\sum_{\ell=0}^{k-1}\frac{(-\xi)^\ell}{\ell!} p_1^{(\ell)}(\xi)\ ,
\end{equation}
where
\[
p_1^{(\ell)}(\xi)=\xi \,\varrho_\ell(\xi) +\ell \,\varrho_{\ell-1}(\xi),\:\:\:\:
\varrho_\ell(\xi) :=
\frac{\diff ^\ell}{\diff \xi^\ell}\left(\frac{1}{\psi(\xi)}\right)\ .
\]
In addition, $\varrho_\ell(\xi)$ can be found recursively from
\[
\varrho_\ell(\xi)= -\frac{1}{\psi(\xi)}\sum_{m=0}^{\ell-1} {{\ell}\choose{m}}\varrho_m(\xi) \psi^{(\ell-m)}(\xi)\ .
\]
\end{proposition}
\begin{proof}
We prove (\ref{E1}) inductively. The validity of the expression for $p_k(\xi)$ is obvious for $k=1$. Now suppose the claim holds for some $k\in{\mathbb N}.$ Then,
\begin{align*}
p_{k+1}(\xi)&= p_{k}(\xi)-\frac{\xi}{k}\,p_k'(\xi)\\
&=\sum_{{\ell=0}}^{k-1}\frac{(-\xi)^\ell}{\ell!} p_1^{(\ell)}(\xi) +\frac{\xi}{k}\sum_{\ell=1}^{k-1}\frac{(-\xi)^{\ell-1}}{(\ell-1)!} p_1^{(\ell)}(\xi) -\frac{\xi}{k}\sum_{\ell=0}^{k-1}\frac{(-\xi)^\ell}{\ell!} p_1^{(\ell+1)}(\xi)\\
&=\sum_{{\ell=0}}^{k-1}\frac{(-\xi)^\ell}{\ell!} p_1^{(\ell)}(\xi) -\frac{\xi}{k}\frac{(-\xi)^{k-1}}{(k-1)!} p_1^{(k)}(\xi)=\sum_{\ell=0}^{k}\frac{(-\xi)^\ell}{\ell!} p_1^{(\ell)}(\xi)\ ,
\end{align*}
where the third step follows by recognizing a telescopic series. This proves the first claim. The second claim follows trivially by $\ell$ times differentiating $\xi/\psi(\xi)$ (where we use the binomial expansion for higher derivatives of products of functions). Regarding the third claim, observe that, for any $\ell\in{\mathbb N}$,
\[
\frac{\diff ^\ell}{\diff \xi^\ell} \left(\psi(\xi)\cdot\frac{1}{\psi(\xi)}\right) =0\ ,
\]
or, equivalently,
\[
\sum_{m=0}^\ell {{\ell}\choose{m}}\varrho_m(\xi) \psi^{(\ell-m)}(\xi)=0\ .
\]
This immediately yields the stated recursion.
\end{proof}

\begin{remark} {A sanity check of the above formula is that it should yield, as a consequence of the celebrated PASTA property (`Poisson arrivals see time averages'; see \cite[Section VII.6.1]{book_A2003}), that
$\lim_{t\to\infty} \P(V(t) = 0) =\lim_{k\to\infty} \P(V_k = 0).$ This relation indeed holds, as follows from
\begin{align}\label{E3}
\lim_{\xi\downarrow 0} p_k(\xi) &= \lim_{\xi\downarrow 0} p_1(\xi) = \lim_{\xi\downarrow 0} \frac{\xi}{\psi(\xi)} = \lim_{\xi\downarrow 0} \frac{1}{\psi'(\xi)} =\frac{1}{\psi'(0)}=\varphi'(0)\ ;
\end{align}
the leftmost expression equals $\lim_{k\to\infty} \P(V_k = 0)$, whereas the rightmost expression, as a direct consequence of the generalized Pollaczek-Khinchine formula equals $\lim_{t\to\infty} \P(V(t) = 0)$ (see \cite[Thm. 3.2]{book_DM2015}). The {second} equality in (\ref{E3}) is due to our explicit expression for $p_1(\xi)$. \hfill$\Diamond$
}
\end{remark}

\begin{remark} {Remarkably, inspecting the proof of Proposition \ref{P1} reveals that this approach provides us with a general devise to translate a transform at an exponential epoch (with mean $\xi^{-1}$) into its counterpart at an Erlang epoch (with parameters $k$ and $\xi$). 
Indeed, defining $\alpha_k(\xi\,|\,s):= \E\, {\rm e}^{-s \bar A_k}$, with $(A_t)_{t\in{\mathbb R}}$ some stochastic process and $\bar A_k:=A_{S_k}$,
we obtain that
\[\alpha_k(\xi\,|\,s) := \sum_{\ell=0}^{k-1}\frac{(-\xi)^\ell}{\ell!} \alpha_1^{(\ell)}(\xi\,|\,s).\]
This idea can e.g.\ be used to find the transform of the workload process in a L\'evy-driven queue, say with spectrally-positive input, at an Erlang epoch. Once more assuming $V_0=0$, in this case we should take \cite[Thm. 4.1]{book_DM2015}
\[
\alpha_1(\xi\,|\,s) = \E\, {\rm e}^{-s \bar A_1}=\frac{\xi}{\psi(\xi)}\cdot\frac{\psi(\xi)-s}{\xi-\varphi(s)}\ .
\]
In this case the evaluation of the derivatives $\alpha_1^{(\ell)}(\xi\,|\,s)$ is more involved than in the setting described in Proposition \ref{P1}; it can be checked that $\alpha_2(\xi\,|\,s)$ agrees with the formula in the last display of \cite[p.\ 54]{book_DM2015} (choosing $x=0$ there), that was derived in an entirely different manner. 
The setup considered in \cite{SBM2016} is related; there the focus is on the case that the inter-event times $T_k$ have different means (where the case with equal means can in principle be dealt with applying a limiting argument).
\hfill$\Diamond$
}\end{remark}

In our setting, we typically have $\varphi(\cdot)$ at our disposal (and all its derivatives), and in addition we can numerically evaluate $\psi(\xi)$ (e.g.\ by bisection), but (as we lack a closed-form expression for $\psi(\xi)$) we cannot easily evaluate the derivatives of $\psi(\xi)$. (An exception is the case that $(X_t)_{t\in{\mathbb R}}$ is a compound Poisson process with exponential jumps; then the inverse $\psi(\cdot)$ allows an explicit expression.) To remedy this, the classical Fa\`a di Bruno formula is helpful. Regarding the first derivative, we know that $\varphi(\psi(\xi))=\xi$, so that by differentiation we obtain $\varphi'(\psi(\xi))\psi'(\xi) =1$, and hence \[\psi'(\xi) = \frac{1}{\varphi'(\psi(\xi))}.\]
Differentiating one more time yields $\varphi''(\psi(\xi))(\psi'(\xi))^2+ \varphi'(\psi(\xi))\psi''(\xi) =0$, leading to
\[
\psi''(\xi) = -\frac{\varphi''(\psi(\xi))(\psi'(\xi))^2}{\varphi'(\psi(\xi))}=\frac{\varphi''(\psi(\xi))}{(\varphi'(\psi(\xi)))^3}\ .
\]
This procedure extends to higher-order derivatives. In general, the Fa\`a di Bruno formula yields that, for any $k\in \{2,3,\ldots\}$,
\[
0 = \frac{\diff ^k}{\diff \xi^k} \xi= \frac{\diff ^k}{\diff \xi^k}\varphi(\psi(\xi)) = \sum_{{\bs m}\in{\mathscr M}_k}\frac{k!}{m_1!\cdots m_k!}\varphi^{(m_1+\cdots +m_k)}(\psi(\xi))\prod_{j=1}^k \left(\frac{\psi^{(j)}(\xi)}{j!}\right)^{m_j}\ ,
\]
where the summation is over the set ${\mathscr M}_k$ containing all non-negative integers $m_1,\ldots,m_k$ such that $m_1 + 2 m_2+\cdots+km_k = k.$ With ${\mathscr M}^\circ_k:={\mathscr M}_k\setminus\{0,\ldots,0,1\}$, we thus obtain the following recursive formula. 
\begin{lemma} \label{L1} 
For $k\in\{2,3,\ldots\}$,
\begin{equation}\label{E2} 
\psi^{(k)}(\xi) = -\frac{1}{\varphi'(\psi(\xi))}\sum_{{\bs m}\in{{\mathscr M}^\circ_k}}\frac{k!}{m_1!\cdots m_k!}\varphi^{(m_1+\cdots +m_k)}(\psi(\xi))\prod_{j=1}^{k-1} \left(\frac{\psi^{(j)}(\xi)}{j!}\right)^{m_j}\ .
\end{equation}
\end{lemma}
The above relation is a genuine recursion, due to the fact that $m_1+\cdots+m_k<k$ for all ${\bs m}\in{{\mathscr M}^\circ_k}$, entailing that when evaluating $\psi^{(k)}(\xi)$ all quantities appearing in the right-hand side of (\ref{E2}) are known. 

\vb

The next step is to compute the $r_k(\xi)$ from the $p_k(\xi)$, which can be done recursively. We partition the event of having a positive workload at observations 1 up $k-1$ (i.e., $V_1>0$, \ldots, $V_{k-1}>0$) and a workload 0 at the $k$-th observation (i.e., $V_k=0$), as follows. The main idea is that the event under consideration can be written as the difference between (A) the event that $V_k=0$, (B) the event that $V_k=0$ but 
$V_\ell=0$ for (at least) some $\ell\in\{1,\ldots,k-1\}.$ The probability of event (A) is $p_k(\xi)$. The event (B) can be written as the union of the {\it disjoint} events 
\[{\mathscr C}_{k.\ell} := \{V_1>0.\ldots,V_{\ell-1}>0,V_\ell = 0,V_k=0\},\]
for $\ell\in\{1,\ldots,k-1\}$; the event ${\mathscr C}_{k.\ell}$ 
has probability $r_\ell(\xi) p_{k-\ell}(\xi)$.
We thus obtain that $r_k(\xi)$ can be evaluated recursively through the following relation, providing us with the distribution of the QBP.
\begin{lemma} 
For $k\in{\mathbb N}$, 
\[r_k(\xi) = p_k(\xi) - \sum_{\ell=1}^{k-1} r_\ell(\xi) p_{k-\ell}(\xi),\]
where the empty sum is defined as $0$. 
Here the probabilities $p_j(\xi)$ directly follow from Proposition~$\ref{P1}$ and Lemma $\ref{L1}$. 
\end{lemma}

\subsection{LRT for quasi busy periods}\label{S32}
In this subsection we point out how the results that were presented in Section~\ref{S31} can be used to develop a test for the setting
(\ref{T1}),
relying on a sample of QBPs $(R_1,\ldots,R_n)$, indicating the number of samples between every two consecutive observations corresponding to a zero-workload.
The log-likelihood of a sample $(R_1,\ldots,R_n)$ is now, in self-evident notation,
\[
\ell_n:=\sum_{i=1}^nP_i,\:\:\:\:P_i:=\log\left( \frac{r_{R_i}^{(0)}(\xi)}{r_{R_i}^{(1)}(\xi)}\right)\ ,
\]
which can be computed relying on the expressions derived in Section \ref{S31}. 
We can now construct standard LRT. In a two-sided test there are two thresholds, say $x_0$ and $x_1$ such that $x_0<0<x_1$ The decision rule is based on $N:=\inf\{n\geqslant 1:\ \ell_n\notin[x_0,x_1]\}$: reject the null hypothesis if $\ell_N>x_1$ and accept the null hypothesis if $\ell_N<x_0$. If $x_0=-\infty$, then this is a power-one test and the type-I error probability is given by
\[
\P_{\mathrm{H}_0}(N<\infty)=\P_{\mathrm{H}_0}\left(\sup_{n\geqslant 1}\ell_n \geqslant x_1\right)\ .
\]
Of course, a power-one test may never stop sampling in cases that ${\rm H}_0$ is never accepted. In some applications this assumption is reasonable because the underlying system works continuously and observations keep being collected. In other cases one may stop the test after some large number of observations, in which case the power of the test will be close to one.

We proceed by reflecting on the pros and cons of this test. 
The main disadvantage of this method is that a substantial amount of information is lost when transforming the workload observations into QBPs: in fact it is only used whether a observation is zero or positive (i.e., its precise value is ignored).
There are, however, two important attractive properties:
\begin{enumerate}
\item The test allows distinguishing between any pair of two L\'evy subordinators with $\varphi_0(\cdot)\neq\varphi_1(\cdot)$ and an arbitrary sampling rate $\xi$. This is because the distribution function $r_k$ depends on all derivatives of the inverse function $\psi(\cdot)$. As we will see later, the test that is presented in Section \ref{sec:cond_LRT} does not have this property. 
\item As mentioned, the sample $(R_1,\ldots,R_n)$ is i.i.d. As a consequence, standard methods for approximating the error {$\alpha(x)$} are readily available; see, e.g., the textbooks \cite{book_S2013,book_Y2013}.
\end{enumerate}

From a practical point of view, as the number of operations required to evaluate $r_i^{(k)}(\xi)$ is of the  order $i!$, there may be computational issues. This means that, particularly for larger $k$, the computations may become time consuming. This problem can be overcome by a truncation: for some appropriately chosen $K$, all QBP durations of at least $K$ are lumped together, and have probability \[1-\sum_{j=1}^{K-1}r_i^{(k)}(\xi).\]

\vb

We conclude this section by providing a quantification of the type-1 error. Consider the random walk $\ell_n=\sum_{i=0}^n P_i$. For any given threshold $x$ we define the first passage time through $N_x:=\inf\{n:\ \ell_n\geqslant x\}$. We are interested in the hitting probability $\alpha(x):=\P_{\mathrm{H}_0}(N_x<\infty)$, and the expected first passage time $\E_{\mathrm{H}_1}N_x$. 

To characterize $\alpha(x)$ we follow a standard procedure. Let $\kappa(\beta):=\log\E_{\mathrm{H}_0}  {\rm e}^{\beta \ell_1}$. Then, for any $\beta$ for which $\kappa(\beta)$ is well-defined, $\exp({\beta \ell_n-n\kappa(\beta)})$ is a mean-1 martingale.
The Lundberg coefficient is given by the $\gamma>0$ being the unique solution of $\kappa(\gamma)=0$ (where existence of this solution follows from $\kappa'(0)=\E_{\mathrm{H}_0} S_1<0$). Applying \cite[Thm. III.5.1]{book_A2003}, we have that
\[
{\alpha(x)}=\P_{\mathrm{H}_0}(N_x<\infty)\approx C {\rm e}^{-\gamma x}\ ,
\]
where $C=\lim_{x\to\infty}\E_{{\rm H}_0} {\rm e}^{-\gamma(N_x-x)}$; an explicit expression for $C$ can be found in e.g.\ \cite{KOR}.
In addition, by standard arguments, $\E_{{\rm H}_1}N_x \approx x/|\E_{{\rm H}_1} P_1|.$

\section{Rate of convergence to steady state}\label{sec:converge}

In the previous section we have set up a test based on quasi busy periods. In Section \ref{sec:cond_LRT} we propose an alternative approach, which we call a conditional likelihood ratio test (CLRT). As a preparation to the performance analysis of this CLRT, in the present section we provide results describing the speed of convergence of a L\'evy-driven storage system. More specifically, they (i) facilitate the computation of the asymptotic variance of the likelihood ratio, and (ii) enable the construction of a functional limit theorem for the likelihood-ratio that can then be used in order to approximate the test's error probability by a Brownian motion hitting probability. Importantly, however, these speed-of-convergence results are, to the best of our knowledge, new, and relevant in their own right. 
 
 \vb

The convergence rate conditions we deal with here are weaker than those required for geometric ergodicity, in the sense of \cite{book_MT2009}. Informally speaking, we do not require that the observed workload process converges exponentially fast to the stationary distribution, but rather that it does so at a rate such that the sum of absolute deviations of certain functions of the workload (relative to their expectations according to the stationary distribution, that is) converges. 

 \vb
 
In what follows we make use of both the continuous time workload $V(t)=X(t)+L(t)$, where $L(t)=-\inf_{0\leqslant s\leqslant t}X(s)$, and the corresponding discretely observed workload process $V_n=V(S_n)$, where $S_n$ is Erlang distributed with scale parameter $n$ and shape parameter $\xi>0$. Recall that $V$ denotes the steady-state workload, where the stability condition $\varphi'(0)>0$ is assumed throughout. The first lemma states a PASTA-type equivalence result for comparing the accumulated deviation from the stationary expectation in continuous time to the corresponding deviation at Poisson epochs. 

\begin{lemma} \label{lemma:PASTA_difference} Assume $\varphi\apost(0)>0$. 
Then, for any initial workload $V(0)$ such that $\P(V(0)<\infty)=1$ and any {measurable} function $g(\cdot)$,
\begin{equation}\label{eq:int_sum}
\xi \int_0^\infty \E[g(V(t))-g(V)|V(0)]\diff t = \sum_{n=1}^\infty \E[g(V_n)-g(V)|V(0)]\ ,
\end{equation}
assuming that both sides of \eqref{eq:int_sum} converge almost surely.
\end{lemma}
\begin{proof} 
We can write the left-hand side of (\ref{eq:int_sum}) as
\[
\xi \int_0^\infty \E[g(V(t))-g(V)|V(0)]\diff t =\lim_{q\downarrow 0} \frac{1}{q} \int_0^\infty q\xi  {\rm e}^{-q \xi t} \,\E[g(V(t))-g(V)|V(0)]\diff t\ ,
\]
which equals, with $T_\xi$ denoting an exponentially distributed random variable with mean $\xi^{-1}$,
\[
\lim_{q\downarrow 0} q^{-1}\cdot {\E[g(V(T_{q\xi}))-g(V)|V(0)]}\ .
\]
Now recall that a geometrically distributed (with success parameter $q$) number of exponentially distributed random variables (with mean $\xi^{-1}$) is exponentially distributed (with mean $(q\xi)^{-1}$). This means that we can rewrite the expression in the previous display as
\[
\lim_{q\downarrow 0} \frac{1}{q} \sum_{n=1}^\infty q(1-q)^{n-1} \E[g(V)-g(V_n)|V(0)]=\sum_{n=1}^\infty  \E[g(V_n)-g(V)|V(0) ]\ ,
\]
which equals the right-hand side of (\ref{eq:int_sum}). 
\end{proof}

\begin{lemma}\label{lemma:Evev}
For any initial workload $V(0)$ such that $\P(V(0)<\infty)=1$,
\begin{align}\nonumber
\E\left[V_1  {\rm e}^{-\alpha V_1}|V(0)\right]&=\frac{\xi}{(\xi-\varphi(\alpha))^2}\Big(\big(V(0)(\xi-\varphi(\alpha))-\varphi\apost(\alpha)\big) {\rm e}^{-\alpha V(0)}\:+\\&\hspace{3.3cm}\frac{\xi-\varphi(\alpha)+\alpha\varphi\apost(\alpha)}{\psi(\xi)} {\rm e}^{-\psi(\xi)V(0)}\Big)\ .
\label{eq:Eveav}
\end{align}
If $0<\varphi\apost(0)<\infty$, then
\begin{equation}\label{eq:Evev}
\E\left[V  {\rm e}^{-\alpha V}\right]=\frac{\varphi\apost(0)(\alpha\varphi\apost(\alpha)-\varphi(\alpha))}{\varphi(\alpha)^2}\ .
\end{equation}
\end{lemma}
\begin{proof}
Observe that
\[\E\left[V_1  {\rm e}^{-\alpha V_1}|V(0)\right] = -\frac{\diff}{\diff\alpha}\E\left[ {\rm e}^{-\alpha V_1}{|V(0)}\right].\]
One obtains \eqref{eq:Eveav} by applying \eqref{eq:V_LST_transient}.
The stationary LST \eqref{eq:Evev} is obtained by taking $\xi\downarrow 0$.
\end{proof}

The following theorem states a number of results that describe the workload (observed at Poisson epochs) convergence to stationarity, conditionally on the initial workload $V(0).$
\begin{theorem}\label{thm:LST_convergence_anyV}
If $\P(V(0)<\infty)=1$, then the following equations hold almost surely: {\rm (i)}~assuming that $X(1)$ has a finite third moment $($i.e., $|\varphi_k^{(3)}(0)|<\infty$$)$, and defining for any $v\geqslant 0$
\[
k_1(v):=\frac{v^2}{2\varphi\apost(0)}+\frac{1}{6}\frac{\varphi^{(3)}(0)}{(\varphi'(0))^2}-\frac{1}{4}\frac{(\varphi^{(2)}(0))^2}{(\varphi'(0))^3} \ ,
\]
we have
\begin{equation}\label{eq:EV_convegence_anyV}
\sum_{n=1}^\infty\left( \E[V_n|V(0)]-\E V\right)= \xi \,k_1(V(0))\ ;
\end{equation}
{\rm (ii)}~defining for any $v,\alpha\geqslant 0$
\[
k_2(v,\alpha):= -\frac{ {\rm e}^{-\alpha v}+\alpha v}{\varphi(\alpha)}+\frac{\alpha}{(\varphi(\alpha))^2}\varphi'(0)+\frac{\alpha}{\varphi(\alpha)}\frac{\varphi^{(2)}(0)}{2\varphi'(0)}\ ,
\]
we have, for any $\alpha\geqslant 0$,
\begin{equation}\label{eq:LST_convegence_anyV}
\sum_{n=1}^\infty\left( \E\left[ {\rm e}^{-\alpha V_n}|V(0)\right]-\E\left[ {\rm e}^{-\alpha V}\right]\right)= \xi \,k_2(V(0),\alpha)\ ;
\end{equation}
{\rm (iii)}~defining for any $v,\alpha\geqslant 0$
\begin{align*}
k_3(v,\alpha)&:= \frac{\varphi'(\alpha)}{\varphi(\alpha)^2} \Big(\alpha v- {\rm e}^{-\alpha v}-\frac{v}{\varphi(\alpha)}(1- {\rm e}^{-\alpha v})\Big)\:-\\&\:\:\:\:\:\:\:\:\:\frac{\varphi\apost(0)}{(\varphi(\alpha))^2}
\Big(1-2\frac{\varphi'(\alpha)}{\varphi(\alpha)}\Big)+\frac{1}{\varphi(\alpha)}\frac{\varphi^{(2)}(0)}{2\varphi'(0)}\Big(1-\frac{\alpha\varphi'(\alpha)}{\varphi(\alpha)}\Big)\ .
\end{align*}
we have, for any $\alpha\geqslant 0$,
\begin{equation}\label{eq:EVeV_convegence_anyV}
\sum_{n=1}^\infty\left( \E\left[V_n {\rm e}^{-\alpha V_n}|V(0)\right]-\E\left[V {\rm e}^{-\alpha V}\right]\right) = \xi \,k_3(V(0),\alpha)\ .
\end{equation}\end{theorem}
\begin{proof} 
We first evaluate the series \eqref{eq:LST_convegence_anyV} for a given $V(0)=v$ by considering the continuous-time analog,
\begin{align*}
d_v(\alpha)&:=\int_0^\infty \left(\E\left[ {\rm e}^{-\alpha V(t)}|V(0)=v\right]-\E\left[ {\rm e}^{-\alpha V}\right]\right)\diff t\\
&= \lim_{q\downarrow 0}\frac{1}{q}\int_0^\infty q {\rm e}^{-q t} \left(\E\left[ {\rm e}^{-\alpha V(t)}|V(0)=v\right]-\E\left[ {\rm e}^{-\alpha V}\right]\right)\diff t\\
&= \lim_{q\downarrow 0}\frac{1}{q}\left(\E\left[ {\rm e}^{-\alpha V(T_q)}|V(0)=v\right]-\E\left[ {\rm e}^{-\alpha V}\right]\right)\ , 
\end{align*}
where $T_q$ is an exponential random variable with rate $q$, and then \eqref{eq:V_LST_transient}  and \eqref{eq:GPK} yield
\[
d_v(\alpha)=\lim_{q\downarrow 0}\frac{1}{q}\left(\frac{q}{q-\varphi(\alpha)}\left( {\rm e}^{-\alpha v}-\frac{\alpha}{\psi(q)} {\rm e}^{-\psi(q)v}\right)-\frac{\alpha\varphi\apost(0)}{\varphi(\alpha)}\right) \ .
\]
The limit can computed by applying L'H\^opital's rule twice and yields
that $
d_v(\alpha$ equals $k_2(v,\alpha)$, as defined above.
By Lemma \ref{lemma:PASTA_difference} we thus obtain
\[
\sum_{n=1}^\infty\left( \E\left[ {\rm e}^{-\alpha V_n}|V(0)=v\right]-\E\left[ {\rm e}^{-\alpha V}\right]\right)=\xi \,d_v(\alpha)\ .
\]
As $\P(V(0)<\infty)=1$ we can condition on $V(0)$ to obtain \eqref{eq:LST_convegence_anyV}:
\[
\sum_{n=1}^\infty \left( \E\left[ {\rm e}^{-\alpha V_n}|V(0)\right]-\E\left[ {\rm e}^{-\alpha V}\right]\right)=\xi\,\E[d_{V(0)}(\alpha)|V(0)]=\xi\,d_{V(0)}(\alpha)= \xi \,k_2(V(0),\alpha)\ .
\]
To establish \eqref{eq:EV_convegence_anyV} and \eqref{eq:EVeV_convegence_anyV} we follow similar arguments, also using Lemma \ref{lemma:Evev}. For any $V(0)=v$, let
\[
e_v:=\int_0^\infty \left(\E\left[V(t)|V(0)=v\right]-\E V\right)\diff t\ ,
\]
and
\[
f_v(\alpha):=\int_0^\infty \left(\E\left[V(t) {\rm e}^{-\alpha V(t)}|V(0)=v\right]-\E\left[V {\rm e}^{-\alpha V}\right]\right)\diff t\ .
\]
Observe that $e_v=-\lim_{\alpha\downarrow 0}d_v'(\alpha)$ and $f_v(\alpha)=- d_v'(\alpha)$. Evaluating the above derivatives and taking the conditional expectation with respect to $V(0)$ immediately yields \eqref{eq:EV_convegence_anyV} and \eqref{eq:EVeV_convegence_anyV}.
\end{proof}

The following technical lemma will play a crucial role in Section \ref{sec:cond_LRT}.
\begin{lemma} \label{lemma: k2_bound} Assume $X(1)$ has a finite second moment $($i.e., $\varphi^{(2)}(0)<\infty$$)$ and let $k_2^\star(v):=\sup_{\alpha>0}|k_2(v,\alpha)|$. Then one of the following two statements holds for any $v\geqslant 0$:
{\rm (i)}~there exists $\alpha^\star\in(0,\infty)$ such that $k_2^\star(v)=|k_2(v,\alpha^\star)|<\infty$;
{\rm (ii)}~$k_2^\star(v)= |\frac{1}{2}\varphi^{(2)}(0)/{\varphi'(0)}-v |<\infty$.
Moreover, in either case we have that $\E[k_2^\star(V)]<\infty$.
\end{lemma}
\begin{proof}
First observe that from \eqref{eq:LST_convegence_anyV} we conclude that $\lim_{\alpha\downarrow 0}k_2(v,\alpha)=0$. Moreover, for any $v,\alpha\geqslant 0$, we can rewrite $k_2(v,\alpha)$ as
\[
k_2(v,\alpha)= \frac{\varphi'(0)\alpha-\varphi(\alpha)( {\rm e}^{-\alpha v}+\alpha v)}{(\varphi(\alpha))^2}+\frac{\alpha}{\varphi(\alpha)}\frac{\varphi^{(2)}(0)}{2\varphi'(0)}\ .
\]
By the definition of $\varphi(\alpha)$ we have that $\lim_{\alpha\to\infty}{\varphi(\alpha)}/{\alpha}=1$, and thus, 
\[
 \frac{\varphi'(0)\alpha-\varphi(\alpha)( {\rm e}^{-\alpha v}+\alpha v)}{(\varphi(\alpha))^2}= \left.{\left(\frac{\varphi'(0)-{\varphi(\alpha)} {\rm e}^{-\alpha v}/{\alpha}}{\alpha}-v\frac{\varphi(\alpha)}{\alpha}\right)}\right/{\left(\frac{\varphi(\alpha)}{\alpha}\right)^2}\xrightarrow{\alpha \to \infty} -v \ .
\]
Therefore, \[\lim_{\alpha\to \infty}|k_2(v,\alpha)|=\left|\frac{\varphi^{(2)}(0)}{2\varphi'(0)}-v\right|<\infty\] for any $v\geqslant 0$. We conclude that, for any $v\geqslant 0$, $|k_2(v,\alpha)|$ is a bounded and continuous function, with respect to $\alpha\in(0,\infty]$, and it therefore admits a maximal value at some $\alpha^\star<\infty$ or approaches an upper bound as $\alpha\to\infty$. Finally, $\E[k_2^\star(V)]<\infty$, as a consequence of the fact that \eqref{eq:V_moments} in combination with the finite second moment assumption implies that $\E V<\infty$.
\end{proof}

For the case of the initial workload $V(0)$ being $0$ we refine the results of Theorem \ref{thm:LST_convergence_anyV} to absolute convergence of the series. This result will be useful in establishing the asymptotic variance of the likelihood ratio in Section \ref{sec:large_sample}.

\begin{lemma}\label{lemma:LST_convergence}
If $X(1)$ has a finite second moment $($i.e, $\varphi^{(2)}(0)<\infty$$)$, then for any $\alpha>0$,
\begin{equation}\label{eq:LST_convegence}
\sum_{n=1}^\infty \left|\E\left[ {\rm e}^{-\alpha V_n}|V(0)=0\right]-\E\left[ {\rm e}^{-\alpha V}\right]\right|
=\xi\left(-\frac{1}{\varphi(\alpha)}+\frac{\alpha}{(\varphi(\alpha))^2}\varphi'(0)+\frac{\alpha}{\varphi(\alpha)}\frac{\varphi^{(2)}(0)}{2\varphi'(0)}\right)
<\infty\ .
\end{equation}
If $X(1)$ has a finite third moment $($i.e, $|\varphi_k^{(3)}(0)|<\infty$$)$, then
\begin{equation}\label{eq:EV_convegence}
\sum_{n=1}^\infty \left|\E\left[ V_n|V(0)=0\right]-\E V \right|
=\xi\left(
\frac{1}{4}\frac{(\varphi^{(2)}(0))^2}{(\varphi'(0))^3}-\frac{1}{6}\frac{\varphi^{(3)}(0)}{(\varphi'(0))^2}
\right)
<\infty\ .
\end{equation}
and, for any $\alpha>0$,
\begin{equation}\label{eq:EVeV_convegence}
\sum_{n=1}^\infty \left|\E\left[V_n {\rm e}^{-\alpha V_n}|V(0)=0\right]-\E\left[V {\rm e}^{-\alpha V}\right]\right|\leqslant \xi\left(\frac{1}{4}\frac{(\varphi^{(2)}(0))^2}{(\varphi'(0))^3}-\frac{1}{6}\frac{\varphi^{(3)}(0)}{(\varphi'(0))^2}
\right)
<\infty\ .
\end{equation}
\end{lemma}
\begin{proof} 
Let $V(t)=X(t)+L(t)$ and $V^\star(t)=X(t)+L^\star(t)$, where $L(t):=-\inf_{0\leqslant s\leqslant t}X(s)$ and $L^\star(t):=\max\{V,L(t)\}$. Recall that $V$ is a random variable that is distributed as the stationary workload (whose transform is given by Eqn.\ \eqref{eq:GPK}), so that $V^\star(t)$ is the workload at time $t$ starting with a stationary workload at time $0$ (implying that $V^\star(t)=_{\rm d} V$ for all $t\geqslant 0$). Note that $V(t)$ is its counterpart, but starting with an empty workload at time $0$. 
Because of $V(0)=0$, we have that $V(t)\leqslant V^\star(t)$ for all $t\geqslant 0$ for every sample path of $X(\cdot)$, and in particular $V(S_n)\leqslant V^\star(S_n)$ for any Erlang distributed sampling time $S_n$. Hence,
\[
\sum_{n=1}^\infty \left|\E\left[V_n|V(0)=0\right]-\E V\right|=-\sum_{n=1}^\infty \left(\E\left[V_n|V(0)=0\right]-\E V\right)\ ,
\]
and applying \eqref{eq:EV_convegence_anyV} for $V(0)=0$ yields \eqref{eq:EV_convegence}. 

Similarly we also have that $ {\rm e}^{-\alpha V(t)}\geqslant  {\rm e}^{-\alpha V^\star(t)}$ for any $\alpha\geqslant 0$ and all $t\geqslant 0$ for every sample path of $X(\cdot)$, hence
\[
\sum_{n=1}^\infty \left|\E\left[ {\rm e}^{-\alpha V_n}|V(0)=0\right]-\E\left[ {\rm e}^{-\alpha V}\right]\right|=\sum_{n=1}^\infty \E\left[ {\rm e}^{-\alpha V_n}|V(0)=0\right]-\E\left[ {\rm e}^{-\alpha V}\right]\ ,
\]
and applying \eqref{eq:LST_convegence_anyV} for $V(0)=0$ yields \eqref{eq:LST_convegence}. 

If $g(\cdot)$ is a Lipschitz continuous function, there exists a constant $K\in(0,\infty)$ such that, for all $t\geqslant 0$,
\[
|g(V(t))-g(V)|\leqslant K\,|V(t)-V| \ .
\]
By Jensen's inequality, the Lipschitz assumption, and Eqn.\ \eqref{eq:int_sum} in Lemma \ref{lemma:PASTA_difference} ,
\begin{align*}\sum_{n=1}^\infty \left|\E\left[ g(V_n)\,|\,V(0)=0\right]-\E\left[g(V)\right]\right|&\leqslant
\sum_{n=1}^\infty \E \left|g(V_n)-g(V)\right|\\
&\leqslant K\sum_{n=1}^\infty \E \left|V_n-V\right| = K\sum_{n=1}^\infty \E \left[V-V_n\right]\\
&= K\xi \int_0^\infty \E \left[V-V(t)\right]\diff t <\infty\ .
\end{align*}
Therefore, the inequality in \eqref{eq:EVeV_convegence} follows by verifying that $x\mapsto x  {\rm e}^{-\alpha x}$ is a function with Lipschitz constant $K=1$.
\end{proof}

\begin{remark}{
In  \cite[Thm. 2]{TT1979} it was shown that the M/G/1 virtual waiting time is geometrically ergodic if and only if $1-G(x)\leqslant c  {\rm e}^{-\mu x}$ for some $c,\mu>0$. In Theorem \ref{thm:LST_convergence_anyV} above we assumed weaker conditions: finiteness of specific moments rather than light-tailed jumps. Note, however, that geometric ergodicity is stronger than the property that we found, because we do not specify a rate of convergence, but just the finiteness of the integral. For the Brownian approximation of the error probability in the hypothesis testing application presented in Section \ref{sec:normal_approx} we will see that a geometric convergence rate is a sufficient, but not  necessary, condition.}\hfill$\Diamond$
\end{remark}

\begin{corollary} \label{C2}
If $X(1)$ has a  finite second moment $($i.e, $\varphi^{(2)}(0)<\infty$$)$, then for any $\alpha>0$,
\begin{equation}\label{eq:LST_convegence}
\sum_{n=1}^\infty \left|\E\left[ {\rm e}^{-\alpha V_n}|V(0)=0\right]-\E\left[ {\rm e}^{-\alpha V}\right]\right|
\leqslant \Xi:= \xi\frac{\varphi^{(2)}(0)}{2(\varphi'(0))^2}
<\infty\ .
\end{equation}
\end{corollary}
\begin{proof} Recall that $\varphi(\alpha)\geqslant \alpha\varphi'(0).$ Observing that
\[-\frac{1}{\varphi(\alpha)}+\frac{\alpha}{(\varphi(\alpha))^2}\varphi'(0)\leqslant 0,\:\:\:\:\frac{\alpha}{\varphi(\alpha)}\frac{\varphi^{(2)}(0)}{2\varphi'(0)}\leqslant \frac{\varphi^{(2)}(0)}{2(\varphi'(0))^2},\]
the claim follows.
\end{proof}

\section{Approach II: conditional likelihood ratio test}\label{sec:cond_LRT}

As in our first approach, in our second approach we use a transform of the observation that allows us to evaluate the likelihood. 
The underlying idea is that 
we construct an LRT based on Bernoulli variables of observing either a zero workload or a positive workload, conditional on the value of the previous workload observation. To this end, we introduce $Y_i:=\mathbf{1}(V_i=0)$ to denote a sequence of idle-period indicators, for $i=1,\ldots,n$. Conditional on ${\bs V}:=(V_0,\ldots,V_n)$, ${\bs Y}=(Y_1,\ldots,Y_n)$ is distributed as a sequence of independent, but {\it not identically distributed}, Bernoulli random variables. More specifically, the corresponding likelihood reads, using (\ref{eq:P_idle_transient}),
\begin{equation}\label{eq:Y_likelihood}
\P_{\psi(\xi)}(Y_1,\ldots,Y_n\,|\,{\bs V})=\prod_{i=1}^n\left[\frac{\xi}{\psi(\xi)} {\rm e}^{-\psi(\xi)V_{i-1}}\right]^{Y_i}\left[1-\frac{\xi}{\psi(\xi)} {\rm e}^{-\psi(\xi)V_{i-1}}\right]^{1-Y_i}\ .
\end{equation}
The likelihood function \eqref{eq:Y_likelihood} depends on the input distribution only through the constant $\theta:=\psi(\xi)\in(\xi,\infty)$. We aim at developing a test for the simple hypothesis testing problem
\[
\begin{array}{cc}
{\rm H}_0: & \theta=\theta_0 \ , \\
{\rm H}_1: & \theta=\theta_1\ .
\end{array} 
\]
From now on we assume that the hypothesis ${\rm H}_0$ and ${\rm H}_1$ correspond to the Laplace exponents $\varphi_0(\cdot)$ and $\varphi_1(\cdot)$ that can be distinguished by the parameters $\theta_0=\psi_0(\xi)$ and $\theta_1=\psi_1(\xi)$, given the sampling rate $\xi$. For a sample of workload observations, ${\bs V}=(V_0,\ldots,V_n)$, the conditional likelihood ratio test (CLRT) is based on the statistic
\[
L_n=\prod_{i=1}^n \frac{\P_{{\rm H}_1}(Y_i\,|\,V_{i-1})}{\P_{{\rm H}_0}(Y_i\,|\,V_{i-1})}\ ,
\]
where $\P_{{\rm H}_k}(Y_i\,|\,V_{i-1})$ can be evaluated using \eqref{eq:P_idle_transient}. 

\vb

Before proceeding to the analysis of this test, a few remarks are in place. As can be seen from the above expressions, the
test can only distinguish between distributions if $\psi_0(\xi)\neq\psi_1(\xi)$ for the current sampling rate $\xi$. This is a minor problem though: one could always choose a sampling rate $\xi$ for which $\psi_0(\xi)$ and $\psi_1(\xi)$ do not coincide. In the three bullets below we consider this issue in greater detail. The requirement that $\psi_0(\xi)\neq\psi_1(\xi)$ is hardly a restriction: for an arbitrary pair of Laplace exponents $(\varphi_0,\varphi_1)$ the $\xi$ such that $\psi_0(\xi)=\psi_1(\xi)$ is effectively a degenerate case. In addition, in many settings the sampling rate $\xi$ can be determined exogenously, so that this issue does not play a role.
\begin{itemize}
\item[$\circ$]
For an M/M/1 queue with known service rate $\mu$, we can test for hypotheses on the arrival rates: $\lambda_0\neq\lambda_1$. There is the advantage that the inverse of the exponent function is known in closed-form:
\[
\psi_k(\xi)=\frac{1}{2}\left(\xi+\lambda_k-\mu+\sqrt{(\xi+\lambda_k-\mu)^2+4\xi\mu}\right),\ \:\:{k=0,1}\ .
\]
This function is clearly monotone in $\lambda$ (for any given $\xi$). A hypothesis on $\lambda$ is therefore equivalent to an hypothesis on $\psi(\xi)$. Conclude that for any $\xi$ the values of $\psi_0(\xi)$ and $\psi_1(\xi)$ do not coincide (as long as $\lambda_0\neq\lambda_1$).
Note that the same argument holds if we replace the roles of arrival and service rates, i.e., fix $\lambda$ and test for $\mu_0\neq\mu_1$. The monotonicity of $\psi(\xi)$ with respect to $\mu$ can be verified by straightforward algebra.
\item[$\circ$] For an M/M/1 queue with unknown arrival and service rates $\lambda$ and $\mu$, there may be values of $\xi$ for which we cannot always test hypotheses on the input intensity $\rho={\lambda}/{\mu}$. For example, $\lambda_1=0.8$ and $\mu_1=1$ yield $\rho_1=0.8$, and $\lambda_0=0.235$ and $\mu_0=2$ yield $\rho_0=0.51<\rho_1$. However, $\psi_1(2)=\psi_2(2)=2.576$. We conclude that in this case even if the traffic intensity is very different, then the conditional likelihood cannot distinguish between the two instances by testing for the parameter $\psi(\xi)$. The obvious remedy is, as mentioned, to pick another sampling rate $\xi.$
\item[$\circ$] Similar considerations play a role for more general L\'evy subordinators. Consider for instance a storage system in which the driving L\'evy process is a Gamma process with parameters $(\beta,\gamma)$. A Gamma process has increments that are, per time unit, $\Gamma(\beta,\gamma)$ distributed. Similar to the M/M/1 case, for any sampling rate $\xi$ a test for either $\beta$ or $\gamma$ can be constructed, but when setting up a test for the traffic intensity $\rho= {\beta}/{\gamma}$ one has to check whether the chosen $\xi$ is such that $\psi_0(\xi)\neq\psi_1(\xi)$; if not, then $\xi$ has to be adapted.
\end{itemize}

\subsection{Power-one sequential test}\label{sec:power1_test}

In the sequential test based on the conditional likelihood ratio, the null hypothesis is rejected when the test statistic attains a high value: for some threshold $x$, reject ${\rm H}_0$ if $L_n\geqslant x$. Note that this test either rejects the null hypothesis or does not terminate. The latter can only occur if ${\rm H_0}$ holds, as $L_n$ will reach any threshold $x<\infty$ with probability one under ${\rm H_1}$. Formally, with $N_x=\inf\{n\geqslant 1:\ L_n\geqslant x\}$, we have $\P_{{\rm H}_1}(N_x<\infty)=1$, making it a power-one sequential test. On the other hand, there is a positive probability that the threshold $x$ will be reached under ${\rm H_0}$, i.e., a false rejection of the null hypothesis. We denote the probability of this type-I error by 
\[
\alpha(x):=\P_{{\rm H}_0}(N_x<\infty)\ ;
\]
our goal is to compute or approximate this quantity. 

Applying \eqref{eq:Y_likelihood} yields
\begin{equation}\label{eq:LRT_n}
\begin{split}
L_n &=\prod_{i=1}^n\frac{\P_{{\rm H}_1}(Y_i\,|\,V_{i-1})}{\P_{{\rm H}_0}(Y_i\,|\,V_{i-1})}=\prod_{i=1}^n \frac{\left(\frac{\xi}{\theta_1} {\rm e}^{-\theta_1 V_{i-1}}\right)^{Y_i}\left(1-\frac{\xi}{\theta_1} {\rm e}^{-\theta_1 V_{i-1}}\right)^{1-Y_i}}{\left(\frac{\xi}{\theta_0} {\rm e}^{-\theta_0 V_{i-1}}\right)^{Y_i}\left(1-\frac{\xi}{\theta_0} {\rm e}^{-\theta_0 V_{i-1}}\right)^{1-Y_i}} \\
& = \left(\frac{\theta_0}{\theta_1}\right)^n\prod_{i=1}^n  {\rm e}^{(\theta_0-\theta_1)V_{i-1}{Y_i}} \left(\frac{\theta_1-\xi  {\rm e}^{-\theta_1V_{i-1}}}{\theta_0-\xi  {\rm e}^{-\theta_0V_{i-1}}}\right)^{1-Y_i}  \ .
\end{split}
\end{equation}

In what follows we work with the log-likelihood ratio (LLR), defined as $\ell_n:=\log L_n$. The LLR $\ell_n$ is more convenient than $L_n$ because of its additive structure. Concretely, we can write $\ell_n=\sum_{i=1}^n Z_i$ for random variables $Z_i$. It should be kept in mind, however, that $\ell_n$ is not a classical random walk, as the increments are neither identically distributed (unless one starts off in stationarity) nor independent. From now on we consider the test with the stopping rule $N_x=\inf\{n\geqslant 1:\ \ell_n\geqslant x\}$. There are two key performance measures for this test. In the first place one commonly considers the significance level 
\[
\alpha(x):=\P_{{\rm H}_0}(N_x<\infty)=\sum_{n=1}^\infty \P_{{\rm H}_0}(N_x=n)\ .
\]
The second key metric is the expected number of samples until rejection of the null hypothesis 
\[
\tau_x:=\E_{{\rm H}_0}N_x\ .
\]
From now we use the short notations $\P_k$ and $\E_k$ for probabilities and expectations under the null hypothesis ($k=0$) and alternative hypothesis ($k=1$).

\subsection{Large sample asymptotics}\label{sec:large_sample}

In the sequel we assume that the workload process starts in stationarity. This makes the increments $Z_i$ identically distributed, but they are obviously not independent. As a consequence,
standard methods from sequential analysis are not directly applicable. Nevertheless, as we will show below, we can provide approximations for the test's performance measures.

In this subsection we establish the asymptotic behavior of the log-likelihood ratio $\ell_n$ as the sample size $n$ grows large. Specifically, we establish a strong law of large numbers (SLLN) for the  {mean log-likelihood-ratio} and a functional central limit theorem (FCLT) for the corresponding centered empirical process, assuming that the workload process starts in stationarity. These results will be used in the sequel to approximate the performance of the CLRT: we approximate $\alpha(x)$ and $\tau_x$ by the hitting probability and expected hitting time, respectively, of a Brownian motion with an appropriately chosen drift and variance.

Taking the logarithm of \eqref{eq:LRT_n} we have $\ell_n=\sum_{i=1}^n Z_i$, where\begin{equation}\label{eq:Zi}
Z_i:=\log\left(\frac{\theta_0}{\theta_1}\right)+(\theta_0-\theta_1) Y_i \,V_{i-1}+(1-Y_i)\log \left(\frac{\theta_1-\xi  {\rm e}^{-\theta_1V_{i-1}}}{\theta_0-\xi  {\rm e}^{-\theta_0V_{i-1}}}\right)\ .
\end{equation}
If the workload process is stable under $\mathrm{H}_k$ for $k\in\{0,1\}$, i.e., $\E_k X(1)=-\varphi_k\apost(0)<0$, then by the PASTA property the stationary distribution $V$ is also the limit of $V_n$ as $n\to\infty$ with respect to ${\rm P}_k$, so that $V_n=_{\rm d} V$. We denote a stationary increment of the LLR by $Z$; because we start off in stationarity, $Z_n=_{\rm d} Z$.

In the following lemma we establish a SLLN for the stationary first and second moment of the sequence $(Z_n)_{n\in{\mathbb N}}$. These will be used later to establish a FCLT for the LLR process.
A complication lies in the aforementioned fact that the $Z_n$ are not independent, so that standard LLN and FCLT techniques cannot be applied directly. To remedy this, we resort in the rest of this subsection to utilizing the special structure of the workload sampled at Poisson times, in combination with a martingale FCLT for weakly dependent stationary random variables.

Define
\[g(v):=\log \left(\frac{\theta_1-\xi  {\rm e}^{-\theta_1v}}{\theta_0-\xi  {\rm e}^{-\theta_0v}}\right),\:\:\:h(v):=\frac{\xi}{\theta_0} {\rm e}^{-\theta_0 v},\:\:\:w_1:=\log\left(\frac{\theta_0}{\theta_1}\right),\:\:\:w_2:=\theta_0-\theta_1 \ ,\]
so that, by (\ref{eq:Zi}),
\[Z_i = w_1 + w_2\, Y_i\,V_{i-1}+(1-Y_i)\,g(V_{i-1})\, .\]
\begin{lemma}\label{lemma:LLN_l0} 
For $k\in\{0,1\}$, if $\varphi_k\apost(0)>0$, then, as $n\to\infty$, {\em (i)}~$\frac{1}{n}\ell_n=\frac{1}{n}\sum_{i=1}^n Z_i\asarrow \E_k Z=m_k$, and {\em (ii)}~$\frac{1}{n}\ell_n^2=\frac{1}{n}\sum_{i=1}^n Z_i^2\asarrow \E_k Z^2= s_k$,
 where
\begin{equation}\label{eq:l0}
m_k:=w_1+w_2\,\E_{k}\left[h(V)\,V\right] +\E_{k}\left[ \left(1-h(V) \right)g(V)\right] \ ,
\end{equation}
and \begin{equation}\label{eq:l02}
\begin{split}
s_k &= w_1^2 +w_2^2\, \E_k \left[h(V)\,V^2\right]+\E_k\left[(1-h(V))\,g^2(V)\right]
\:+ \\ 
&\quad 2w_1w_2\,\E_k\left[h(V)\,V\right]+2w_1\,\E_k\left[(1-h(V))\,g(V)\right] +2w_2\,\E_k\left[(1-h(V))\,g(V)\,V\right] \ .
\end{split}
\end{equation}
\end{lemma}
\begin{proof}
We start by proving claim (i). Applying PASTA, we have that $\frac{1}{n}\ell_n=\frac{1}{n}\sum_{i=1}^nZ_i \asarrow \E_k Z$. The stationary sample average can be computed  as follows. Recalling that $V(0)=_{\rm d} V$, the mean of the first increment $Z_1$ (and hence also the mean of all other increments) equals
\[
\begin{split}
\E_k[\E_k[Z_1|V_0]] &= w_1+w_2\,\E_k[Y_1\,V_0]+\E_k\left[(1-Y_1)\,g(V_0)\right] \\ 
&= w_1+w_2\,\E_k[\P_k[Y_1=1|V_0]\,V_0]+\E_k\left[\P_k[Y_1=0|V_0]\,g(V_0)\right] \\
&=w_1+w_2\,\E_{k}\left[h(V)\,V\right] +\E_{k}\left[ \left(1-h(V) \right)g(V)\right]\ ,
\end{split}
\]
where in the last equality \eqref{eq:P_idle_transient} has been used.
Claim (ii) follows in the same manner.
\end{proof}

We next turn our attention to the asymptotic distribution of $\sqrt{n}\left(\ell_n/n-m_k\right)=\left(\ell_n-nm_k\right)/\sqrt{n}$ as $n\to \infty$, for $k=0,1$.
Let $z_i:=\E_k Z_i$ and $M_n:=\sum_{i=1}^n(Z_i-z_i)$. Then 
 \begin{equation}\label{eq:ell_n_decomposition}
\ell_n-nm_k=\sum_{i=1}^n(Z_i-m_k+z_i-z_i)=M_n+\sum_{i=1}^n(z_i-m_k)\ .
\end{equation}
As $V_0$ (and $Z_1$) is stationary, we have $z_i=m_k$ for all $i$, so that we can focus on deriving the limiting distribution of $M_n/\sqrt{n}$ as $n\to\infty$. Let BM$(d, \sigma^2)$ be a Brownian motion with drift $d$ and variance coefficient $\sigma^2$, i.e., at time $t$ having a normal distribution with mean $dt$ and variance $\sigma^2 t$. The main result of this subsection is the following FCLT. 

\begin{theorem}\label{thm:ln_FCLT}
If $\varphi_k>0$ and $|\varphi_k^{(3)}(0)|<\infty$, and the initial workload is stationary $($i.e., $V_0=_{\rm d} V$$)$, then as $n\to\infty$, for $k=0,1$,
\begin{equation}\label{eq:ln_FCLT} \left(
\frac{1}{\sqrt{n}}\left(\ell_{\floor{nt}}-nt\,m_k\right)\right)_{t\geqslant 0}\darrow \mathrm{BM}\left(0,\sigma_k^2\right)\ ,
\end{equation}
{with respect to the Skorohod topology on the functional space $D[0,\infty)$ $($see, e.g., \cite[Ch.\ 3]{book_Billingsley1999}$)$,} where 
$\sigma_k^2:=\lim_{n\to\infty}n^{-1}\,{\Var_k\,M_n} <\infty.$\end{theorem}

The proof of this theorem relies on the methodology presented in \cite[Ch.\ 18--19]{book_Billingsley1999}. In particular, we decompose $M_n$ into a martingale difference process and an additional term that vanishes when scaled by $1/\sqrt{n}$, and then apply the FCLT for stationary ergodic martingale difference processes (see \cite[Thm.\ 18.3]{book_Billingsley1999}). The decomposition is similar to the one used in the proof of \cite[Thm.\ 19.1]{book_Billingsley1999}, but we utilize the specific structure of our problem and do not impose the stronger assumptions that are required there. Informally, this essentially boils down to showing that the dependence between the increments in the LLR diminishes sufficiently fast. To this end, we apply the results of Section \ref{sec:converge} that describe the convergence rate of the LLR to stationarity. This will then be used to verify the FCLT conditions of the martingale approximation, and in particular to show that $\sigma_k^2<\infty$.

We proceed by sketching the structure of the remainder of this subsection, geared towards proving Theorem \ref{thm:ln_FCLT}. 
\begin{itemize}
\item[$\circ$]
Lemma \ref{lemma:FCLT_condition}, which builds on Theorem \ref{thm:LST_convergence_anyV} and Lemma \ref{lemma: k2_bound}, establishes that the LLR converges fast enough to satisfy a sufficient condition for the FCLT of Eqn.\ \eqref{eq:ln_FCLT}. 
\item[$\circ$]Then, in Proposition~\ref{prop:ln_sigma}, we rely on Lemma \ref{lemma:LST_convergence} and Corollary \ref{C2} to show that $\sigma_k^2:=\lim_{n\to\infty}n^{-1}\,{\Var_kM_n}$ is finite, and moreover the absolute convergence of the series of covariance terms $(\Cov_k(Z_1,Z_{i+1}))_{i=0,1,\ldots}$. 
\item[$\circ$] Then we are in a position to prove Theorem \ref{thm:ln_FCLT}. We use the martingale decomposition and apply Lemma \ref{lemma:FCLT_condition} to verify that $M_n$ can be written as a sum of a martingale difference process $\tilde{M}_n$ with stationary increments that have a finite second moment (thus satisfying the conditions of \cite[Thm.\ 18.3]{book_Billingsley1999}) and an error term $R_n$ that is almost surely finite. Then Proposition \ref{prop:ln_sigma} is used to show that the asymptotic variance of $\tilde{M}_n/\sqrt{n}$ equals $\sigma_k^2$. 
\end{itemize}The proofs of Lemma \ref{lemma:FCLT_condition} and Proposition \ref{prop:ln_sigma} are quite lengthy and are therefore relegated to the appendix.

\begin{lemma}\label{lemma:FCLT_condition} Let $k\in\{0,1\}$.
If $\varphi_k\apost(0)>0$, and the initial workload is stationary $($i.e., $V_0=_{\rm d} V$$)$, then the following claims hold: {\em (a)}~With respect to $\P_k$ we have that almost surely,
\begin{equation}\label{eq:dZ_V0}
\sum_{n=1}^\infty (\E_k[Z_n\,|\,V_0]-m_k)=z(V_0)\ ,
\end{equation}
where
\begin{equation}\label{eq:zV0}
\begin{split}
z(v) &:= \frac{\xi^2}{\theta_k} k_3(v,\theta_k) +\xi\sum_{j=1}^\infty\frac{1}{j}\left(\left(\frac{\xi}{\theta_0}\right)^j k_2(v, \theta_0 j) -\left(\frac{\xi}{\theta_1}\right)^j k_2(v, \theta_1 j)\right) \\
&\quad - \frac{\xi^2}{\theta_k} \sum_{j=1}^\infty\frac{1}{j}\left(\left(\frac{\xi}{\theta_0}\right)^j k_2(v, \theta_0 j+\theta_k) -\left(\frac{\xi}{\theta_1}\right)^j k_2(v, \theta_1 j+\theta_k)\right)\ .
\end{split}
\end{equation}
{\rm (b)} $|z(V_0)|<\infty$ almost surely. {\rm (c)} If in addition $|\varphi_k^{(3)}(0)|<\infty$, then $\E_k[z(V)^2]<\infty$.
\end{lemma}

The following proposition establishes that the variance of the LLR is finite, and in addition that the series of covariance terms $(\Cov_k(Z_1,Z_{i+1}))_{i=0,1,\ldots}$ converges absolutely. The second part of this statement is important not just for establishing the FCLT approximation but also from a computational perspective. It implies that the covariance series, and subsequently $\sigma_k^2$, can be evaluated efficiently by means of truncation as the remainder of the series vanishes 
 in absolute terms.

\begin{proposition}\label{prop:ln_sigma}
Let $k\in\{0,1\}$. If $\varphi_k\apost(0)>0$ and $|\varphi_k^{(3)}(0)|<\infty$, and the initial workload is stationary $($i.e., $V_0=_{\rm d} V$$)$, then as $n\to\infty$, 
\begin{equation}\label{eq:var_Mn}
\sigma_k^2=\lim_{n\to\infty}n^{-1}{\Var_kM_n} =s_k-m_k^2+2\sum_{i=1}^\infty c_{ki}\ ,
\end{equation}
where $c_{ki}:=\Cov_k(Z_1,Z_{i+1})$ is given by
\begin{equation}\label{eq:cov_ki}
\begin{split}
c_{ki} &= w_2^2\left(\E_k\left[V_0 Y_1 V_{i}Y_{i+1}\right]-\E_k^2\left[V_{0}Y_1\right] \right) \\
&\ +w_2\left(\E_k\left[V_0 Y_1 g(V_{i})(1-Y_{i+1})\right]-\E_k\left[V_0 Y_1\right]\E_k\left[(1-Y_{1})g(V_{0})\right] \right) \\
&\ + w_2\left(\E_k\left[g(V_0)(1-Y_1)V_i Y_{i+1}\right] -{ \E_k\left[V_0 Y_1\right]\E_k\left[(1-Y_{1})g(V_{0})\right]}\right)\\
&\ +\E_k\left[ g(V_{0})(1-Y_{1})g(V_{i})(1-Y_{i+1})\right]-\E_k^2\left[g(V_0) (1-Y_{1})\right]\ .
\end{split}
\end{equation}
Moreover, $\sum_{i=1}^\infty |c_{ki}|<\infty$ and therefore $\sigma_k^2<\infty$.
\end{proposition}

\begin{proof}[Proof of Theorem \ref{thm:ln_FCLT}]
In {Proposition} \ref{prop:ln_sigma} it is established that \[\lim_{n\to\infty} n^{-1}\Var_kM_n={\sigma^2_k}<\infty;\] this is, however, a necessary but not sufficient condition for the FCLT \eqref{eq:ln_FCLT} to hold. We will show that $M_n$ can be decomposed into two terms; one that is almost surely finite and the other has the same asymptotic distribution as a martingale difference sum that satisfies the conditions of the martingale FCLT \cite[Thm.\ 18.3]{book_Billingsley1999}.

The decomposition relies on iterating expectations, as follows. Define ${\mathscr F}_n:=\sigma\{V_0,\ldots,V_n\}$ and consider
\begin{align*}
Z_i-m_k&=
(Z_i-\E_k[Z_i\,|\,V_{i-1}])+(\E_k[Z_i\,|\,V_{i-1}]-m_k)\ .
\end{align*}
Then observe that the first term, i.e., $D_{i,1}:=Z_i-\E_k[Z_i\,|\,V_{i-1}]$, is a ${\mathscr F}_{i-1}$-martingale difference. We therefore have $\E_k[D_{i,1}|{\mathscr F}_{i-1}]=0$. Similarly, the second term can be decomposed into
\[
\E_k[Z_i\,|\,V_{i-1}]-m_k=(\E_k[Z_i\,|\,V_{i-1}]-\E_k[\E_k[Z_i\,|\,V_{i-1}]|\,V_{i-2}])+(\E_k[\E_k[Z_i\,|\,V_{i-1}]|\,V_{i-2}]-m_k)\ ,
\]
where the first term in the right-hand side of the previous display
\[
D_{i,2}:=\E_k[Z_i\,|\,V_{i-1}]-\E_k[\E_k[Z_i\,|\,V_{i-1}]|\,V_{i-2}]=\E_k[Z_i\,|\,V_{i-1}]-\E_k[Z_i\,|\,V_{i-2}]\ , 
\]
is now a ${\mathscr F}_{i-2}$-martingale difference. Continuing along these lines, we readily obtain, for $j\in\{1,\ldots,i\}$, that
\begin{equation}\label{eq:Dij}
D_{i,j}:=\E_k[Z_i\,|\,V_{i-j+1}]-\E_k[Z_i\,|\,V_{i-j}]\ ,
\end{equation}
is a ${\mathscr F}_{i-j}$-martingale difference. By performing $i$ iterations, we thus find that
\[
Z_i-m_k=\sum_{j=1}^iD_{i,j}+(\E_k[Z_i\,|\,V_0]-m_k)\ ,
\]
which implies that
\begin{equation}\label{prev}
M_n=\sum_{i=1}^n (Z_i-m_k)=\sum_{i=1}^n\sum_{j=1}^iD_{i,j}+\sum_{i=1}^n(\E_k[Z_i\,|\,V_0]-m_k)\ .
\end{equation}
Let us consider the first term in the right-hand side of (\ref{prev}). By changing the order of summation and some relabelling, we obtain
\[
\sum_{i=1}^n\sum_{j=1}^iD_{i,j}=\sum_{i=1}^n\sum_{j=0}^{i-1} D_{i,i-j}=\sum_{j=0}^{n-1}\sum_{i=j+1}^{n} D_{i,i-j}=\sum_{j=0}^{n-1}\sum_{i=j+1}^{\infty} D_{i,i-j}-\sum_{j=0}^{n-1}\sum_{i=n+1}^{\infty} D_{i,i-j} \ ,
\]
and observe that $D_{i,i-j}=\E_k[Z_i\,|\,V_{j+1}]-\E_k[Z_i\,|V_{j}]$ implies that $\sum_{i=j+1}^{\infty} D_{i,i-j}$ is a ${\mathscr F}_{j}$-martingale difference. Note that for now we have assumed that the infinite series are finite almost surely, and we will later verify that this is indeed the case. The stationarity of $V_0$ implies that for any $i> j$,
\[
D_{i,i-j}=\E_k[Z_i\,|\,V_{j+1}]-\E_k[Z_i\,|\,V_{j}]=_{\rm d} \E_k[Z_{i-j}\,|\,V_{1}]-\E_k[Z_{i-j}\,|\,V_{0}]=D_{i-j,i-j}\ .
\]
Therefore, for every $j\in{\mathbb N}$,
\[
D_j:=\sum_{i=j+1}^{\infty} D_{i,i-j}=_{\rm d} \sum_{i=j+1}^{\infty} D_{i-j,i-j}=\sum_{i=1}^{\infty} D_{i,i}\ ,
\]
with $(D_j)_{j\in{\mathbb N}}$ being a sequence of stationary random variables. Now Lemma \ref{lemma:FCLT_condition}(a) implies that, almost surely,
\[
\begin{split}
\sum_{i=1}^{\infty} D_{i,i} &=\sum_{i=1}^{\infty} (\E_k[Z_{i}\,|\,V_{1}]-\E_k[Z_{i}\,|\,V_{0}])=\sum_{i=1}^{\infty} (\E_k[Z_{i}\,|\,V_{1}]-m_k)-\sum_{i=1}^{\infty} (\E_k[Z_{i}\,|\,V_{0}]-m_k) \\
&= \E_k[Z_{1}\,|\,V_{1}]-m_k+z(V_1)-z(V_0)=Z_{1}-m_k+z(V_1)-z(V_0) \ .
\end{split}
\] 
Under the assumption of $|\varphi_k^{(3)}(0)|<\infty$, Lemma \ref{lemma:FCLT_condition}(c) further implies that $\E_k[z(V_0)^2]<\infty$. By applying similar arguments $\E_k[z(V_0)z(V_1)]<\infty$ as well (in  particular, note that one can condition on $V_0$ and use \eqref{eq:V_LST_transient} and \eqref{eq:P_idle_transient} to obtain a decomposition of $\E_k[z(V_0)z(V_1)]$ as linear and quadratic terms of $V_0$). Therefore, $\E_k[D_j^2]<\infty$.
Applying the telescopic structure in \eqref{eq:Dij} yields
\[
\begin{split}
\sum_{j=0}^{n-1}\sum_{i=n+1}^{\infty} D_{i,i-j} &= \sum_{j=0}^{n-1}\lim_{m\to\infty}\sum_{i=n+1}^{m} D_{i,i-j}=\lim_{m\to\infty}\sum_{j=0}^{n-1}\sum_{i=n+1}^{m} D_{i,i-j} = \lim_{m\to\infty}\sum_{i=n+1}^{m}\sum_{j=0}^{n-1} D_{i,i-j} \\
&= \lim_{m\to\infty}\sum_{i=n+1}^{m}\left(\E_k[Z_i|V_n]-\E_k[Z_i|V_0] \right)= \sum_{i=n+1}^{\infty}\left(\E_k[Z_i|V_n]-\E_k[Z_i|V_0] \right)\\
&= \sum_{i=n+1}^{\infty}\left(\E_k[Z_i|V_n]-m_k \right)-\sum_{i=n+1}^{\infty}\left(\E_k[Z_i|V_0]-m_k \right) \ .
\end{split}
\] 
Upon combining the above, we find from (\ref{prev}) that we can write
\[
M_n=\tilde{M}_n +R_n\ ,
\]
where
\[\tilde{M}_n:=\sum_{j=0}^{n-1} D_j,\:\:\:\:\:R_n:=\sum_{i=1}^\infty (\E_k[Z_i\,|\,V_0]-m_k)-\sum_{i=n+1}^{\infty}\left(\E_k[Z_i|V_n]-m_k \right)\ ,\] 
Observe that $\tilde{M}_n$ is a martingale difference sum. Also, $R_n$ is almost surely finite due to Lemma \ref{lemma:FCLT_condition}(b), so that $R_n/\sqrt{n}\to 0$ almost surely as $n\to\infty$. Combining the above, we conclude 
\begin{equation}\label{eq:lim_Mn}
\lim_{n\to\infty}\frac{M_n}{\sqrt{n}}=_{\rm d}\lim_{n\to\infty}\frac{\tilde{M}_n}{\sqrt{n}}=_{\rm d}
\lim_{n\to\infty}\frac{1}{\sqrt{n}}\sum_{j=0}^{n-1}D_j;
\end{equation}
notice that the right-hand side involves a series of the stationary martingale differences $(D_j)_{j\in{\mathbb N}}$. As $\E_k[D_j^2]<\infty$, the conditions of the martingale FCLT \cite[Thm.\ 18.3]{book_Billingsley1999} are satisfied and we conclude that \eqref{eq:ln_FCLT} holds. 

Furthermore, $\E_k[D_j^2]=\lim_{n\to\infty}{n}^{-1}\Var_k\,\tilde{M}_n$. Using the Cauchy-Schwarz inequality
\[
\frac{1}{n}\Cov_k(M_n,R_n)\leqslant \sqrt{\frac{1}{n}\Var_k\,M_n\,\cdot\frac{1}{n}\Var\,R_n}\ ,
\]
together with ${n}^{-1}\Var_k\,M_n\to \sigma_k^2$ (see Lemma \ref{prop:ln_sigma}) and ${n}^{-1}\Var_k\,R_n\to 0$, we conclude that as $n\to \infty$,
\[
\frac{1}{n}\Var_k\tilde{M}_n=\frac{1}{n}\Var_k (M_n-R_n)=\frac{1}{n}\Var_k M_n+\frac{1}{n}\Var_k R_n-\frac{2}{n}\Cov_k(M_n,R_n)\to \sigma_k^2 \ .
\] This concludes the proof.
\end{proof}

\subsection{Brownian approximation}\label{sec:normal_approx}
A standard way to approximate $\alpha(x)$ is by relying on a Brownian approximation, in our case facilitated
by Theorem \ref{thm:ln_FCLT}. Applying this theorem, with $(B(t))_{t\geqslant 0}$ a standard Brownian motion, realizing that $m_0<0$,
\begin{align*}
\alpha(x) &=\P_{0}\left(\sup_{t\geqslant 0}\ell_{\floor{nt}} >x\right)= \P_0\left(\exists t\geqslant 0: \frac{\ell_{\floor{nt}} -nt\,m_0 }{\sqrt{n}}\geqslant \frac{x-nt\,m_0 }{\sqrt{n}}\right)\\
&\approx \P_0\left(\exists t\geqslant 0: \sigma_0 B(t)\geqslant \frac{x-nt\,m_0 }{\sqrt{n}}\right)\\
&=\P_0\left(\exists t\geqslant 0: B(t)+\frac{\sqrt{n}m_0}{\sigma_0}t\geqslant \frac{x}{\sigma_0\sqrt{n}}\right)=
\exp\left(-2\frac{m_0 x}{\sigma^2_0}\right),
\end{align*}
where the last expression is a standard result for the maximum of a Brownian motion with negative drift. Similarly, we can approximate the expected rejection time if the null hypothesis is wrong by
\[
\E_1 N_x \approx \frac{x}{|m_1|}\ .
\]

\subsection{Change of measure and approximation}\label{sec:CLRT_approx}
Apart from a Brownian approximation, another standard technique to approximately evaluate $\alpha(x)$ is by using exponential-change-of-measure techniques. Again, the key difficulty is the intricate dependence between increments of the log-likelihood process.
As before, $\ell_n=\sum_{i=1}^n Z_i$, where the increments $Z_i$ are given by \eqref{eq:Zi}, and $N_x=\inf\{n\geqslant 1:\ \sum_{i=1}^n Z_i\geqslant x\}$. Let the cumulant generating function be given by
\[
\kappa_n(\beta):=\frac{1}{n}\log\E_0  {\rm e}^{\beta \ell_n}=\frac{1}{n}\log\E_0  {\rm e}^{\beta \sum_{i=1}^n Z_i}\ .
\]
Then $M_n(\beta):=\exp({\beta \ell_n-n\kappa_n(\beta)})$ is a martingale with respect to ${\mathscr F}_n:=\sigma\{V_0,\ldots,V_n\}$ such that $\E M_n(\beta)=1$ for all $n\in{\mathbb N}$. Let $\gamma_n$ denote the sequence of solutions to $\kappa_n(\gamma_n)=0$, and consider the change of measure
\[
\P^{(\gamma_n)}(A)=\E_0[ {\rm e}^{\gamma_n \ell_n}\mathbf{1}(A)], \:\:\: A\in\mathscr{F}_n \ .
\]

\begin{lemma}\label{lemma:kappa_Z}
If $\varphi_0\apost(0)>0$ and the initial workload is stationary $($i.e., $V_0=_{\rm d} V$$)$, then
\begin{equation}\label{eq:kappa_Z}
\kappa_1(\beta)=\beta w_1+ \log\left(\frac{\xi(\beta\theta_1+(1-\beta)\theta_0)\varphi_0'(0)}{\theta_0\varphi_0(\beta\theta_1+(1-\beta)\theta_0)}+\frac{1}{\theta_0}\E_0\left[(\theta_1- {\rm e}^{-\theta_1 V})^\beta (\theta_0- {\rm e}^{-\theta_0 V})^{1-\beta}\right]\right)\ .
 \end{equation} 
\end{lemma}
\begin{proof}
Again, applying {PASTA} and the SLLN, in combination with $V_0=_{\rm d} V$, yields the stated.
First observe that
\[\kappa_1(\beta)= \log \E_0  {\rm e}^{\beta(w_1+w_2 v_0Y_1+(1-Y_1)g(V_0))},\]
which equals
\[\beta w_1 +\log\left(\E_0 \left( {\rm e}^{\beta w_2 V_0}h(V_0)+\E_0\left( {\rm e}^{\beta g(V_0)}(1-h(V_0))\right)\right)\right),\]
which, by (\ref{eq:GPK}), equals 
\eqref{eq:kappa_Z}.
\end{proof}

We now point out how $\alpha(x)$ can be approximated in the regime that the sampling rate is very slow, i.e., $\xi\to 0$. Then 
the $Z_i$ are (almost) independent and identically distributed (as the random variable $Z$). Hence, $M_n(\beta)$ roughly equals $\exp({\beta \ell_n-n\kappa_1(\beta)})$, where $\kappa_1(\beta)$ is given by~\eqref{eq:kappa_Z}. The Lundberg coefficient $\gamma>0$ is the solution to $\kappa_1(\gamma)=0$, which is unique because $\kappa_1\apost(0)=\E_{0} Z_1=m_0<0$. Consider the change of measure
\[
\P^{(\beta)}(A)=\E_0[M_n(\beta)\mathbf{1}(A)], \ \:\:A\in\mathscr{F}_n \ .
\]
Then applying  \cite[Thm.\ III.5.1]{book_A2003} we have that, with $C$ defined as $\lim_{x\to\infty}\E_{\gamma} {\rm e}^{-\gamma(N_x-x)}$,
\[
\alpha(x):=\P_{\mathrm{H}_0}(N_x<\infty)\approx C {\rm e}^{-\gamma x}\ ;
\]
a more explicit characterization of $C$ is given in  \cite{KOR}.

{For higher sampling rates the iid stationary approximation may not be satisfactory. Therefore, we now investigate to what extent the correlation between the $Z_i$ can be taken into account.}
Ideally, we would like to relate the stationary coefficient $\gamma$ to the sequence of coefficients $\gamma_n$ that satisfies $\kappa_n(\gamma_n)=0$. 
Analyzing the sequence $\gamma_n$ is challenging because it involves the joint distribution of $(Z_1,\ldots,Z_n)$. The following procedure can be followed.
Let $h_n(\beta):=\E_0  {\rm e}^{\beta \sum_{i=1}^n Z_i}$, and observe that $\kappa_n(\beta)=0$ is equivalent to $h_n(\beta)=1$. We have that
\[
h_2(\beta)=\E_0\left[ {\rm e}^{\beta Z_1+\beta Z_2}\right]=\mathrm{Cov}_0\left( {\rm e}^{\beta Z_1}, {\rm e}^{\beta Z_2}\right)+\E_0\left[ {\rm e}^{\beta Z_1}\right]\E_0\left[ {\rm e}^{\beta Z_2}\right]\ ,
\]
and as $Z_1$ and $Z_2$ both have the same marginal (and stationary) distribution we can write
\begin{equation}\label{eq:h2}
h_2(\beta)=\mathrm{Cov}_0\left( {\rm e}^{\beta Z_1}, {\rm e}^{\beta Z_2}\right)+(h_1(\beta))^2\ .
\end{equation}
Eqn.\ \eqref{eq:h2} can be generalized to
\begin{equation}\label{eq:hn}
\begin{split}
h_n(\beta) &= \mathrm{Cov}_0\left( {\rm e}^{\beta \sum_{i=1}^{n-1}Z_i}, {\rm e}^{\beta Z_n}\right)+h_1(\beta)h_{n-1}(\beta) \\
&= \mathrm{Cov}_0\left( {\rm e}^{\beta \sum_{i=1}^{n-1}Z_i}, {\rm e}^{\beta Z_n}\right)+h_1(\beta)\left(\mathrm{Cov}_0\left( {\rm e}^{\beta \sum_{i=1}^{n-2}Z_i}, {\rm e}^{\beta Z_{n-1}}\right)+h_1(\beta)h_{n-2}(\beta)\right) \\
& \ \ \vdots \\
&= \sum_{j=2}^{n}h_{1}(\beta)^{n-j}\mathrm{Cov}_0\left( {\rm e}^{\beta \sum_{i=1}^{j-1}Z_i}, {\rm e}^{\beta Z_j}\right)+h_1(\beta)^n \ .
\end{split}
\end{equation}

The conclusion is that the sequence $\gamma_n$ can be numerically approximated by simulating the solutions to \eqref{eq:hn}. Such a simulated sequence of solutions is illustrated in Figure \ref{fig:gamma_n}. In Section~\ref{sec:simulation} we explore the performance of the error approximations based on $\gamma_1$ and $\gamma_{20}$.

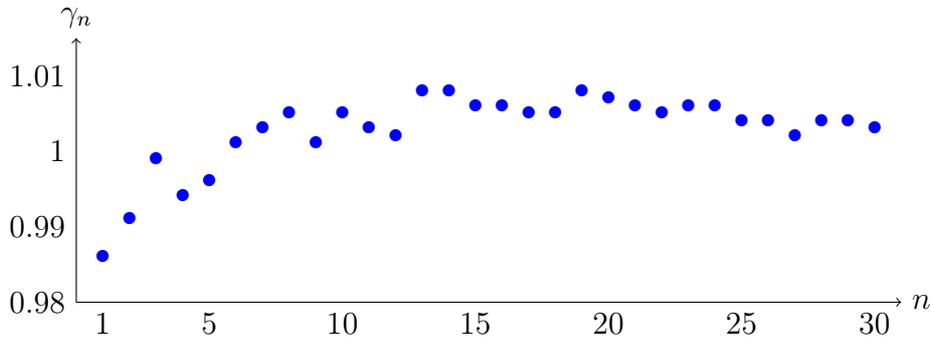
\begin{figure}[H]
\centering
\begin{tikzpicture}[xscale=0.35,yscale=100]
 \def\xmin{0}
 \def\xmax{31}
 \def\ymin{0.98}
 \def\ymax{1.015}
  \draw[->] (\xmin,\ymin) -- (\xmax,\ymin) node[right] {$n$} ;
  \draw[->] (\xmin,\ymin) -- (\xmin,\ymax) node[above] {$\gamma_n$} ;
  \foreach \x in {1,5,10,15,20,25,30}
  \node at (\x,\ymin) [below] {\x};
  \foreach \y in {0.98,0.99,1,1.01}
    \node at (0,\y) [left] {\y};

    \foreach \Point in {( 1 , 0.986 ), ( 2 , 0.991 ), ( 3 , 0.999 ), ( 4 , 0.994 ), ( 5 , 0.996 ), ( 6 , 1.001 ), ( 7 , 1.003 ), ( 8 , 1.005 ), ( 9 , 1.001 ), ( 10 , 1.005 ), ( 11 , 1.003 ), ( 12 , 1.002 ), ( 13 , 1.008 ), ( 14 , 1.008 ), ( 15 , 1.006 ), ( 16 , 1.006 ), ( 17 , 1.005 ), ( 18 , 1.005 ), ( 19 , 1.008 ), ( 20 , 1.007 ), ( 21 , 1.006 ), ( 22 , 1.005 ), ( 23 , 1.006 ), ( 24 , 1.006 ), ( 25 , 1.004 ), ( 26 , 1.004 ), ( 27 , 1.002 ), ( 28 , 1.004 ), ( 29 , 1.004 ), ( 30 , 1.003 )}
    {\node[blue] at \Point {\textbullet};}
    
\end{tikzpicture}
\caption{A solution via simulation to the sequence of solutions $\gamma_n$ to the change of measure equation. The alternative hypothesis were M/M/1 queues with $\lambda_0=0.6$ and $\lambda_1=0.8$ and $\mu_0=\mu_1=10$. The sampling rate is $\xi=3$.}
\label{fig:gamma_n}
\end{figure}

\section{Simulation analysis}\label{sec:simulation} %
In this section we present numerical results on tests corresponding to an M/M/1 queue, distinguishing between 
two values of the arrival rate  $\lambda$  ($\lambda_0=0.6$ and $\lambda_1=0.8$), for a given value of the service rate ($\mu=10$). `QPBT' refers to the test based on the quasi busy periods (Approach~I), and  `CLRT' to the conditional likelihood ratio test (Approach II). 

From our experiments, the following observations can be made:
\begin{itemize}
\item[$\circ$] The QBPT and CLRT tests presented in Sections \ref{sec:QBP} and \ref{sec:cond_LRT}  agree in over 80\% of the sample-path realizations. That is to say, in the vast majority of all cases, if one test rejects, then so does the other. The thresholds for both tests were chosen so that they have approximately the same significance level $\alpha(x)$. 
\item[$\circ$]  Both tests greatly outperform the na\"{\i}ve mean test $n^{-1}\sum_{i=1}^n V_i\geqslant x$ that was discussed in Remark \ref{R1}. This na\"{\i}ve  test has a very lower power for the same significance levels.
\item[$\circ$]  All approximation techniques presented in this paper perform reasonably well. See for example Figure \ref{fig:alpha_lrt_approx} that compares the type-I error probability of the CLRT, based on simulation with three approximations: the Brownian approximation of Section \ref{sec:normal_approx}, the change-of-measure approximation of Section \ref{sec:CLRT_approx} using $\gamma_1$,  and the change-of-measure approximation of Section \ref{sec:CLRT_approx} using $\gamma_{20}$.
\item[$\circ$]  In Figure \ref{fig:alpha_lrt_zeros} the type-I error probabilities are plotted for increasing sampling rates $\xi$. The thresholds are chosen so that the approximated error probability equals 5\%. 
For the QBPT, $x=-{\gamma}^{-1}\log (0.05)$, where $\gamma$ corresponds to the change-of-measure approximation presented in Section \ref{sec:QBP}.
For the CLRT, $x=-{\gamma_{20}}^{-1}\log (0.05)$, where $\gamma_{20}$ corresponds to the change-of-measure approximation presented in Section \ref{sec:CLRT_approx}. We observe that the approximation improves as the sampling rate increases.
\item[$\circ$] In Figure \ref{fig:beta_lrt_zeros} the type-II error probabilities $\beta$ are plotted for increasing sampling rates $\xi$, using the same parameters as in Figure \ref{fig:alpha_lrt_zeros}. The tests were truncated at $n=10^3$. For a low sampling rate, the truncation has no effect; we observe that $\beta\approx 1$, as is the case for $n\to\infty$. However, for higher sampling rates, many more observations are required, and in addition the test does not have power one. This is particularly true for the QBPT, were the power drops to almost 70\%, while for the CLRT it remains above 90\%. 
\item[$\circ$] Figure \ref{fig:tau_lrt_zeros} plots the corresponding expected number of samples until rejection (conditional on rejection before truncation). The QBPT test generally rejects faster, but as we saw in Figure \ref{fig:beta_lrt_zeros} this comes with a lower power, meaning that if the test was not truncated these expectations will be much higher. 
\end{itemize}

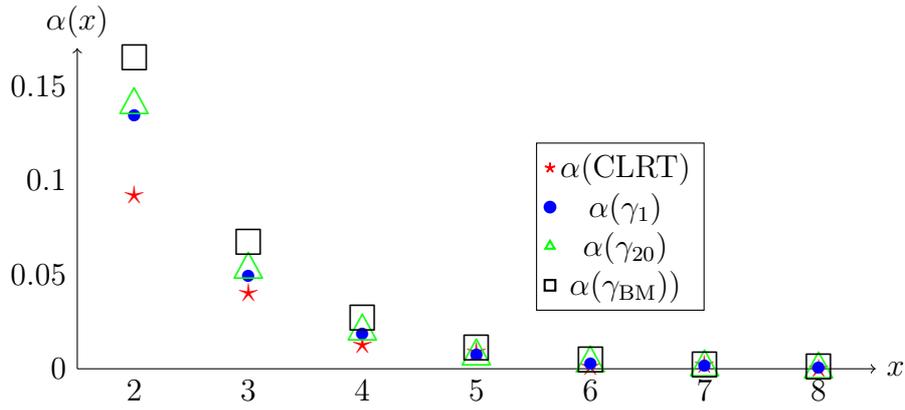
\begin{figure}[h]
\centering
\begin{tikzpicture}[xscale=1.5,yscale=25]
 \def\xmin{1.5}
  \def\xmax{8.5}
  \def\ymin{0}
  \def\ymax{0.17}
    \draw[->] (\xmin,\ymin) -- (\xmax,\ymin) node[right] {$x$} ;
    \draw[->] (\xmin,\ymin) -- (\xmin,\ymax) node[above] {$\alpha(x)$} ;
    \foreach \x in {2,3,4,5,6,7,8}
    \node at (\x,\ymin) [below] {\x};
    \foreach \y in {0,0.05,0.1,0.15}
    \node at (\xmin,\y) [left] {\y};
    
    \foreach \Point in { ( 2 , 0.092 ), ( 3 , 0.04 ), ( 4 , 0.013 ), ( 5 , 0.009 ), ( 6 , 0.001 ), ( 7 , 0.002 ), ( 8 , 0 )}	
    {\node[red] at \Point {\scalebox{1.2}{$\star$}};}

    \foreach \Point in {( 2 , 0.134 ), ( 3 , 0.049 ), ( 4 , 0.018 ), ( 5 , 0.007 ), ( 6 , 0.002 ), ( 7 , 0.001 ), ( 8 , 0 )}
    {\node[blue] at \Point {\textbullet};}
    
    \foreach \Point in {( 2 , 0.14 ), ( 3 , 0.053 ), ( 4 , 0.02 ), ( 5 , 0.007 ), ( 6 , 0.003 ), ( 7 , 0.001 ), ( 8 , 0 )}
    {\node[green] at \Point  {\scalebox{1.2}{$\triangle$}};}
    
    \foreach \Point in {( 2 , 0.165 ), ( 3 , 0.067 ), ( 4 , 0.027 ), ( 5 , 0.011 ), ( 6 , 0.005 ), ( 7 , 0.002 ), ( 8 , 0.001 )}
    {\node[black] at \Point  {\scalebox{1.2}{$\square$}};}

    \begin{customlegend}
    [legend entries={ $\alpha(\mathrm{CLRT})$, $\alpha(\gamma_1)$, $\alpha(\gamma_{20})$, $\alpha(\gamma_{\mathrm{BM}}))$},
    legend style={at={(7,0.12)}}]    
    \addlegendimage{red,mark=star,only marks,thick } 
    \addlegendimage{blue,mark=*,only marks,thick} 
    \addlegendimage{green,mark=triangle,only marks,thick }
    \addlegendimage{mark=square,only marks,thick }
    \end{customlegend}
\end{tikzpicture}
\caption{Probability $\alpha$(CLRT) of type-I errors as a function of the threshold $x$ for the CLRT, together with three approximations. 
The (non-approximative) probabilities were computed using simulation with $n=10^3$ samples.
 The sampling rate is $\xi=3$.}
\label{fig:alpha_lrt_approx}
\end{figure}

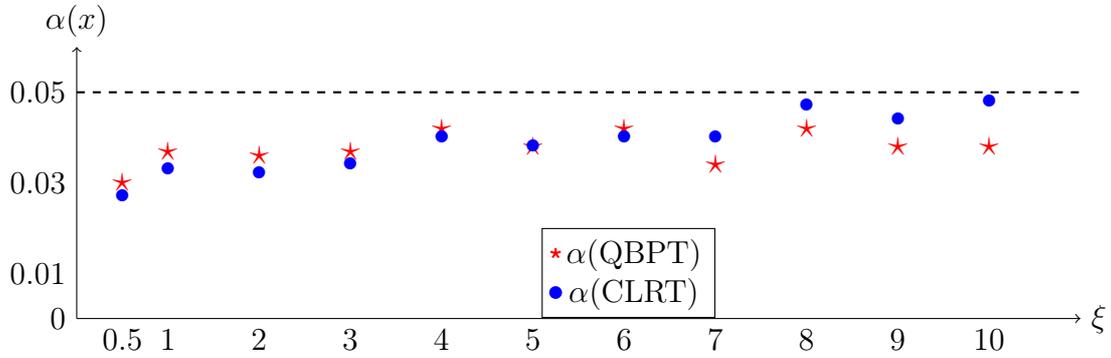
\begin{figure}[H]
\centering
\begin{tikzpicture}[xscale=1.2,yscale=60]
  \def\xmin{0}
  \def\xmax{11}
  \def\ymin{0}
  \def\ymax{0.06}
    \draw[->] (\xmin,\ymin) -- (\xmax,\ymin) node[right] {$\xi$} ;
    \draw[->] (\xmin,\ymin) -- (\xmin,\ymax) node[above] {$\alpha(x)$} ;
    \foreach \x in {0.5,1,2,3,4,5,6,7,8,9,10}
    \node at (\x,\ymin) [below] {\x};
    \foreach \y in {0,0.01,0.03,0.05}
    \node at (0,\y) [left] {\y};
    
    \foreach \Point in { ( 0.5 , 0.03 ), ( 1 , 0.037 ), ( 2 , 0.036 ), ( 3 , 0.037 ), ( 4 , 0.042 ), ( 5 , 0.038 ), ( 6 , 0.042 ), ( 7 , 0.034 ), ( 8 , 0.042 ), ( 9 , 0.038 ), ( 10 , 0.038 )}	
    {\node[red] at \Point {\scalebox{1.2}{$\star$}};}

    \foreach \Point in {( 0.5 , 0.027 ), ( 1 , 0.033 ), ( 2 , 0.032 ), ( 3 , 0.034 ), ( 4 , 0.04 ), ( 5 , 0.038 ), ( 6 , 0.04 ), ( 7 , 0.04 ), ( 8 , 0.047 ), ( 9 , 0.044 ), ( 10 , 0.048 )}
    {\node[blue] at \Point {\textbullet};}
    
    \draw[black, thick, dashed] (0,0.05)--(\xmax,0.05);

    \begin{customlegend}
    [legend entries={$\alpha(\mathrm{QBPT})$, $\alpha(\mathrm{CLRT})$},
    legend style={at={(7,0.02)}}]    
    \addlegendimage{red,mark=star,only marks,thick } 
    \addlegendimage{blue,mark=*,only marks,thick} 
    \end{customlegend}
\end{tikzpicture}
\caption{Probability of type-I errors using both approaches ($\alpha$(CLRT) and $\alpha$(QBPT)). 
The (non-approximative) probabilities were computed using simulation with $n=10^3$ samples.}
\label{fig:alpha_lrt_zeros}
\end{figure}

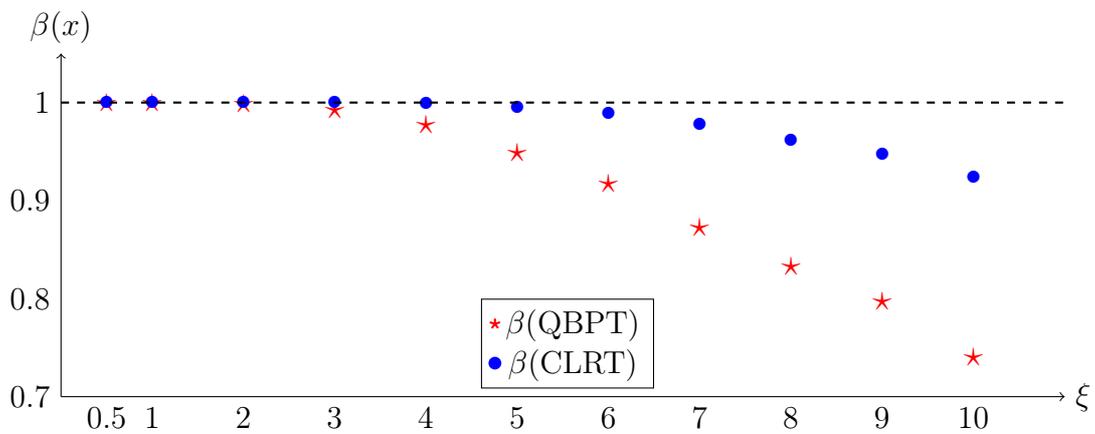
\begin{figure}[H]
\centering
\begin{tikzpicture}[xscale=1.2,yscale=13]
  \def\xmin{0}
  \def\xmax{11}
  \def\ymin{0.7}
  \def\ymax{1.05}
    \draw[->] (\xmin,\ymin) -- (\xmax,\ymin) node[right] {$\xi$} ;
    \draw[->] (\xmin,\ymin) -- (\xmin,\ymax) node[above] {$\beta(x)$} ;
    \foreach \x in {0.5,1,2,3,4,5,6,7,8,9,10}
    \node at (\x,\ymin) [below] {\x};
    \foreach \y in {0.7,0.8,0.9,1}
    \node at (0,\y) [left] {\y};
    
    \foreach \Point in { ( 0.5 , 1 ), ( 1 , 1 ), ( 2 , 0.999 ), ( 3 , 0.992 ), ( 4 , 0.977 ), ( 5 , 0.949 ), ( 6 , 0.917 ), ( 7 , 0.872 ), ( 8 , 0.833 ), ( 9 , 0.797 ), ( 10 , 0.74 )}	
    {\node[red] at \Point {\scalebox{1.2}{$\star$}};}

    \foreach \Point in {( 0.5 , 1 ), ( 1 , 1 ), ( 2 , 1 ), ( 3 , 1 ), ( 4 , 0.999 ), ( 5 , 0.995 ), ( 6 , 0.988 ), ( 7 , 0.977 ), ( 8 , 0.961 ), ( 9 , 0.947 ), ( 10 , 0.923 )}
    {\node[blue] at \Point {\textbullet};}
    
    \draw[black, thick, dashed] (0,1)--(\xmax,1);

    \begin{customlegend}
    [legend entries={$\beta(\mathrm{QBPT})$, $\beta(\mathrm{CLRT})$},
    legend style={at={(6.5,0.8)}}]    
    \addlegendimage{red,mark=star,only marks,thick } 
    \addlegendimage{blue,mark=*,only marks,thick} 
    \end{customlegend}
\end{tikzpicture}
\caption{
Probability of type-II errors using both approaches ($\beta$(CLRT) and $\beta$(QBPT)).
The (non-approximative) probabilities were computed using simulation with $n=10^3$ samples.}
\label{fig:beta_lrt_zeros}
\end{figure}

\begin{figure}[h]
\centering
\begin{tikzpicture}[xscale=1.2,yscale=0.006]
  \def\xmin{0}
  \def\xmax{11}
  \def\ymin{0}
  \def\ymax{550}
    \draw[->] (\xmin,\ymin) -- (\xmax,\ymin) node[right] {$\xi$} ;
    \draw[->] (\xmin,\ymin) -- (\xmin,\ymax) node[above] {$\tau(x)$} ;
    \foreach \x in {0.5,1,2,3,4,5,6,7,8,9,10}
    \node at (\x,\ymin) [below] {\x};
    \foreach \y in {0,100,200,300,400,500}
    \node at (0,\y) [left] {\y};
    
    \foreach \Point in { ( 0.5 , 77.948 ), ( 1 , 114.726 ), ( 2 , 183.3 ), ( 3 , 250.039 ), ( 4 , 305.414 ), ( 5 , 362.659 ), ( 6 , 402.015 ), ( 7 , 444.47 ), ( 8 , 469.136 ), ( 9 , 501.973 ), ( 10 , 527.425 )}	
    {\node[red] at \Point {\scalebox{1.2}{$\star$}};}

    \foreach \Point in {( 0.5 , 64.018 ), ( 1 , 78.144 ), ( 2 , 119.817 ), ( 3 , 161.546 ), ( 4 , 202.176 ), ( 5 , 250.258 ), ( 6 , 286.345 ), ( 7 , 329.28 ), ( 8 , 352.947 ), ( 9 , 392.918 ), ( 10 , 417.662 )}
    {\node[blue] at \Point {\textbullet};}
    
    

    \begin{customlegend}
    [legend entries={$\tau(\mathrm{QBPT})$, $\tau(\mathrm{CLRT})$},    legend style={at={(8.5,180)}}]    
    \addlegendimage{red,mark=star,only marks,thick } 
    \addlegendimage{blue,mark=*,only marks,thick} 
    \end{customlegend}
\end{tikzpicture}
\caption{
Expected number of samples until rejection of the null-hypothesis (under $\mathrm{H}_1$), using both approaches ($\tau$(CLRT) and $\tau$(QBPT)). The (non-approximative) probabilities were computed using simulation with $n=10^3$ samples.}
\label{fig:tau_lrt_zeros}
\end{figure}
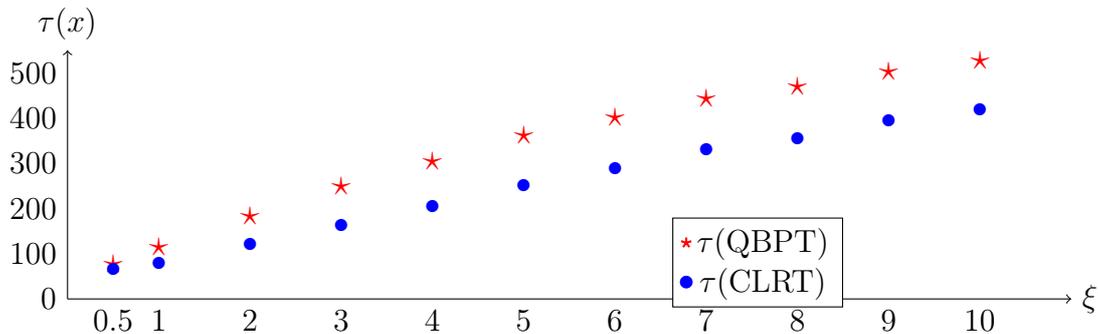

\section{Discussion and concluding remarks}\label{sec:DCR}
This paper has focused on hypothesis testing for L\'evy-driven storage systems. Tests are developed that are capable of distinguishing a Laplace exponent $\varphi_0(\cdot)$ from an alternative Laplace exponent $\varphi_1(\cdot).$ In the setup considered the driving L\'evy process is a subordinator minus a deterministic drift. In both approaches proposed, we exploit the property that for such L\'evy processes the workload level is zero with positive probability. An open challenge is to set up tests for more general L\'evy input. We remark that reflected Brownian motion is not covered by the current framework, but it can be dealt with due to the fact that for this specific model the likelihood can be evaluated in closed form. 

In both approaches information is `lost': in the first approach it is only used whether a workload sample is positive or zero (thus leading to i.i.d.\ quasi busy periods), in the second approach it is only used whether a workload sample is 0 conditional on the previous observation. One thus wonders whether one could develop a test in which `more information is used'. Due to the lack of an explicit expression for the joint density of $V_1,\ldots, V_n$, inevitably, either some information provided by the sample will be lost, or the likelihood of the sample is to be evaluated in an approximate manner.

An important extension of this work is to consider composite hypotheses of the form
\[
\begin{array}{cc}
{\rm H}_0: & \varphi \in \Phi_0 \ , \\
{\rm H}_1: & \varphi \in \Phi_1 \ ,
\end{array} 
\]
where $\Phi_0,\Phi_1$ are disjoint collections of Laplace exponents corresponding to possible L\'evy input distributions. In this case one can construct a generalized likelihood-ratio test (GLRT), for some $x>0$,
\[
L_n:= \frac{\sup_{\varphi \in \Phi_1}\P_{\varphi}(V_{1},\ldots, V_n)}{\sup_{\varphi \in \Phi_0}\P_{\varphi}(V_{1},\ldots, V_n)}\ \geqslant x\ ,
\]
where the likelihood $\P_{\varphi}$ is given by the probability measure associated with the input process with exponent $\varphi$. Note that the likelihoods can correspond to either of the two methods presented in this work. The GLRT poses two main challenges. First of all, the practical problem of evaluating the supremum terms. In our quasi-busy-period approach the likelihood terms are not explicit and can only be evaluated numerically using the recursive formulas (involving the derivatives of $\varphi$ and its inverse) presented in Section \ref{sec:QBP}. Therefore, maximizing these functions on an elaborate function space seems challenging. If $\Phi_i$, $i=1,2$, are parametric families, e.g., $\Phi_0$ corresponds to Compound Poisson and $\Phi_1$ corresponds to Gamma processes, then ad hoc computational methods for the GLRT computation can potentially be constructed. For the second approach, namely the conditional likelihood ratio test, this task becomes easier because the value of the suprema can be evaluated by applying the MLE method of \cite{RBM2019}. The second challenge is the asymptotic performance analysis of the GLRT. As we saw, this is not straightforward even in the context of simple hypothesis. Standard results may apply to the QBPT method due to the iid observations, although this method may be hard to implement in practice. For the CLRT the issue of weak dependence will have to be again carefully addressed for the composite case.

Poisson sampling has evident computational advantages, as pointed out in our paper, but in some practical settings one may prefer equidistant sampling. The techniques presented in \cite{SBM2016} could potentially be used to approximate the corresponding density, but computational challenges are anticipated. Alternatively, the likelihood can be evaluated by numerical inversion of (\ref{eq:V_LST_transient}); this concerns an inversion to obtain a density from a transform as well as an inversion to translate the exponentially distributed time into its deterministic counterpart. In a related setting, such (double) inversion techniques have been explored in \cite{AdIM}.

\section{Appendix: Proofs for Section \ref{sec:large_sample}}

\begin{proof}[Proof of Lemma \ref{lemma:FCLT_condition}] First realize that in our setting $\varphi_k\apost(0)>0$ (i.e., stability of the storage system) implies
$\P_k(V_0<\infty)=1$. For ease, we leave out the subscript $k$ throughout the proof.
(a) By \eqref{eq:Zi} and \eqref{eq:l0}, $\sum_{n=1}^\infty (\E[Z_i\,|\,V_0]-m_k)=d(V_0)+e(V_0)+f(V_0)$, where
\begin{align*}
d(V_0)&: =  \sum_{n=1}^\infty \big(\E\left[Y_{n}V_{n-1}\,|\,V_0\right]-\E\left[Y_1V_{0}\right]\big)\ ,
\\
e(V_0)&:=\sum_{n=1}^\infty \big(\E\left[ g(V_{n-1})\,|\,V_0\right]-\E\left[g(V_0) \right]\big) \ ,
\\
f(V_0)&:=-\sum_{n=1}^\infty \big(\E\left[ Y_{n}g(V_{n-1})\,|\,V_0\right]-\E\left[Y_{1}g(V_0) \right]\big) \ .
\end{align*}
We deal with the three terms separately.

{\it --~First term.} Recall that $Y_n=\mathbf{1}(V_n=0)$ and $\P(V_n=0|V_{n-1})=\frac{\xi}{\theta_k} {\rm e}^{-\theta_k V_{n-1}}$ by \eqref{eq:P_idle_transient}. Thus,
\[
\E\left[Y_{n}V_{n-1}\,|\,V_0\right]=\frac{\xi}{\theta_k}\E\left[V_{n-1} {\rm e}^{-\theta_k V_{n-1}}\,|\,V_0\right]\ , 
\]
and similarly
\[
\E\left[Y_1V_{0}\right]=\frac{\xi}{\theta_k}\E\left[V_{0} {\rm e}^{-\theta_k V_{0}}\right]\ ,
\]
yielding
\[
d(V_0) = \frac{\xi}{\theta_k} \sum_{n=1}^\infty \big(\E\left[V_{n-1} {\rm e}^{-\theta_k V_{n-1}}\,|\,V_0\right]-\E\left[V_{0} {\rm e}^{-\theta_k V_{0}}\right]\big)\ .
\]
Therefore, by Theorem \ref{thm:LST_convergence_anyV}, $d(V_0) = ({\xi^2}/{\theta_k})\cdot k_3(V_0,\theta_k)$.

{\it --~Second term.} The function $g(\cdot)$ can be written, by a Taylor expansion, as
\begin{equation}\label{eq:g_Taylor}
g(v)=\frac{1}{w_1}+\sum_{j=1}^\infty\frac{1}{j}\left(\left(\frac{\xi}{\theta_0}\right)^j {\rm e}^{-\theta_0 j v}-\left(\frac{\xi}{\theta_1}\right)^j {\rm e}^{-\theta_1 j v}\right)\ .
\end{equation}
Therefore,
\[
\begin{split}
e(V_0)&= \sum_{j=1}^\infty\frac{1}{j}\left(\frac{\xi}{\theta_0}\right)^j\sum_{n=1}^\infty\left(\E\left[ {\rm e}^{-\theta_0 j V_{n}}\,|\,V_0\right]-\E\left[ {\rm e}^{-\theta_0 j V}\right]\right) \\
&\quad -\sum_{j=1}^\infty\frac{1}{j}\left(\frac{\xi}{\theta_1}\right)^j\sum_{n=1}^\infty\left(\E\left[ {\rm e}^{-\theta_1 j V_{n}}\,|\,V_0\right]-\E\left[ {\rm e}^{-\theta_1 j V}\right]\right) \ ,
\end{split}
\]
and by \eqref{eq:LST_convegence_anyV} (part (ii) of Theorem \ref{thm:LST_convergence_anyV}),
\[
e(V_0)= \xi\sum_{j=1}^\infty\frac{1}{j}\left(\left(\frac{\xi}{\theta_0}\right)^j k_2(V_0, \theta_0 j) -\left(\frac{\xi}{\theta_1}\right)^j k_2(V_0, \theta_1 j)\right) \ .
\]

{\it --~Third term.} To deal with the third term, we combine the arguments used for the previous two terms. Firstly, conditioning on $Y_{n-1}$, we obtain
\[
f(V_0)=-\frac{\xi}{\theta_k}\sum_{n=1}^\infty \big(\E\left[  {\rm e}^{-\theta_k V_{n-1}}g(V_{n-1})\,|\,V_0\right]-\E\left[ {\rm e}^{-\theta_k V_{0}}g(V_0) \right]\big) \ .
\]
Using the Taylor expansion of $g(\cdot)$, this yields 
\[
\begin{split}
f(V_0)&=-\frac{\xi}{\theta_k} \sum_{j=1}^\infty\frac{1}{j}\left(\frac{\xi}{\theta_0}\right)^j\sum_{n=1}^\infty\left(\E\left[ {\rm e}^{-(\theta_0 j+\theta_k) V_{n}}\,|\,V_0\right]-\E\left[ {\rm e}^{-(\theta_0 j+\theta_k)V}\right]\right) \\
&\quad -\sum_{j=1}^\infty\frac{1}{j}\left(\frac{\xi}{\theta_1}\right)^j\sum_{n=1}^\infty\left(\E\left[ {\rm e}^{-(\theta_1 j+\theta_k) V_{n}}\,|\,V_0\right]-\E\left[ {\rm e}^{-(\theta_1 j+\theta_k) V}\right]\right) \ .
\end{split}
\]
Then, by \eqref{eq:LST_convegence_anyV} (part (ii) of Theorem \ref{thm:LST_convergence_anyV}),
\[
f(V_0)=- \frac{\xi^2}{\theta_k} \sum_{j=1}^\infty\frac{1}{j}\left(\left(\frac{\xi}{\theta_0}\right)^j k_2(V_0, \theta_0 j+\theta_k) -\left(\frac{\xi}{\theta_1}\right)^j k_2(V_0, \theta_1 j+\theta_k)\right) \ .
\]
Combining all of the above, we conclude \eqref{eq:dZ_V0}.

(b) By the triangle inequality we have $|z(V_0)|\leqslant |d(V_0)|+|e(V_0)|+|f(V_0)|$. If $\P(V_0<\infty)=1$, then \eqref{eq:EVeV_convegence_anyV} (part (iii) of Theorem \ref{thm:LST_convergence_anyV}) implies that $|k_2(V_0,\theta)|<\infty$ and $|k_3(V_0,\theta)|<\infty$ for any $\theta\in(0,\infty)$, so that $|d(V_0)|<\infty$. Furthermore, denoting, for $k=0,1$,
\[
a_k:=\sum_{j=1}^\infty\frac{1}{j}\left(\frac{\xi}{\theta_k}\right)^j,
\]
we have $a_k<\infty$ (using that ${\xi}/{\theta_k}<1$), and
\[
e(V_0)= \xi a_0\sum_{j=1}^\infty\frac{1}{j a_0}\left(\frac{\xi}{\theta_0}\right)^j k_2(V_0, \theta_0 j) -\xi a_1\sum_{j=1}^\infty\frac{1}{j a_1}\left(\frac{\xi}{\theta_1}\right)^j k_2(V_0, \theta_1 j) \ .
\]
For the first sum we apply Lemma \ref{lemma: k2_bound}:
\[
\begin{split}
\left|\sum_{j=1}^\infty\frac{1}{j a_0}\left(\frac{\xi}{\theta_0}\right)^j k_2(V_0, \theta_0 j)\right| &\leqslant \sum_{j=1}^\infty\frac{1}{j a_0}\left(\frac{\xi}{\theta_0}\right)^j \left|k_2(V_0, \theta_0 j)\right| \\
&\leqslant \sum_{j=1}^\infty\frac{1}{j a_0}\left(\frac{\xi}{\theta_0}\right)^j k_2^\star(V_0)=k_2^\star(V_0) \ .
\end{split}
\]
The same argument can be used for the second sum in $e(V_0)$. Then the triangle inequality yields
$
|e(V_0)|\leqslant \xi(a_0+a_1) k_2^\star(V_0),$
which is finite almost surely if $\P(V_0<\infty)=1$. The same argument yields that $|f(V_0)|\leqslant A k_2^\star(V_0)<\infty$ for some positive constant $A$ almost surely.

(c) Finally, we will verify that the $z(V)$ has a finite second moment by considering
\[
\begin{split}
z(V)^2&=|z(V)|^2\leqslant (|d(V)|+|e(V)|+|f(V)|)^2 \\
&= |d(V)|^2+|e(V)|^2+|f(V)|^2+2\,|d(V)|\cdot|e(V)|+2\,|d(V)|\cdot|f(V)|+2\,|e(V)|\cdot|f(V)| \ .
\end{split}
\]
Using the upper bounds established in part (b) we have that
\[
z(V)^2\leqslant A_1 k_3(V_0,\theta_k)^2+A_2k_2^\star(V_0)^2+A_3k_2^\star(V_0)\left|k_3(V_0,\theta_k)\right|\ ,
\]
for some positive and finite constants $A_1,A_2,A_3$. Hence, by the definitions of $k_2(v,\alpha)$ and $k_3(v,\alpha)$ in Theorem \ref{thm:LST_convergence_anyV} and $k_2^\star(V_0)$ in Lemma \ref{lemma: k2_bound}, $z(V)^2$ is upper bounded by a linear combination of terms of the form $V$, $V^2$, $ {\rm e}^{-\alpha V}$, $V {\rm e}^{-\alpha V}$ and $V^2 {\rm e}^{-\alpha V}$ (with $\alpha\in(0,2\max\{\theta_0,\theta_1,\alpha^\star\}]$). Therefore, appealing to \eqref{eq:V_moments}, if $\varphi\apost(0)>0$ and $|\varphi_k^{(3)}(0)|<\infty$, then $\E[z(V)^2]<\infty$.
\end{proof}

\begin{proof}[Proof of Proposition \ref{prop:ln_sigma}]
If $Z_0$ is stationary then $Z_i=_{\rm d} Z_0$ for any $i\in{\mathbb N}$. The first claim then follows from the standard identity
\[\Var_k\,M_n = n\,\Var_k\,Z_1+2\sum_{i=1}^{n-1}(n-i)\,\Cov_k(Z_1,Z_{i+1}) \ ;\]
the decomposition of $c_{ki}$ in \eqref{eq:cov_ki} is obtained by plugging in \eqref{eq:Zi}, and using the assumption that $V_0$ is stationary.

It is left to prove that $\sum_{i=1}^\infty |c_{ki}|<\infty$, which we show by using that $V_n$ converges fast enough to its stationary distribution, relying on Lemma \ref{lemma:LST_convergence}. Below we denote the density of the workload level at sample $n$ by $f_{n}(\cdot)$; likewise, the density of the conditional workload at sample $n$ given an initial workload $V_0=v_0$ is denoted by by $f_n(\cdot\,|\,v_0)$. We treat the four individual terms in the right-hand side of (\ref{eq:cov_ki}) separately. For ease, we leave out the subscript $k$ throughout the proof.

{\it --~First term.} Recall that $Y_n=\mathbf{1}(V_n=0)$. Thus,
\[
\E\left[V_0 Y_1 V_{i}Y_{i+1}\right]=\E\left[V_0V_{i}\mathbf{1}(V_{i+1}=0)\mathbf{1}(V_1=0)\right]= \E\left[V_0V_{i}\mathbf{1}(V_{i+1}=0)\,|\,V_1=0\right]\P(V_1=0)\ .
\]
As $V_{i+1}$ and $V_0$ are independent conditional on $V_1$ for any $i\in{\mathbb N}$, we have that
\[
\begin{split}
\E\left[V_0 Y_1 V_{i}Y_{i+1}\right]&= \E\left[V_0\,|\,V_1=0\right]\P(V_1=0)\E\left[V_{i}\mathbf{1}(V_{i+1}=0)\,|\,V_1=0\right] \\ 
&= \E\left[V_0\mathbf{1}(V_1=0)\right]\,\E\left[V_{i}\mathbf{1}(V_{i+1}=0)\,|\,V_1=0\right] \ .
\end{split}
\]
As a consequence, the first term of \eqref{eq:cov_ki} can be written as
\[
w_2^2\,\E\left[V_0\mathbf{1}(V_1=0)\right]a_i \ ,\:\:\:
a_i:=\E\left[V_{i}\mathbf{1}(V_{i+1}=0)\,|\,V_1=0\right]-\E\left[V_0\mathbf{1}(V_1=0)\right]\ .
\]
Next, applying \eqref{eq:P_idle_transient} yields, for $i\in{\mathbb N}$,
\[
\begin{split}
a_i &= \int_0^\infty v\,\P(V_{i+1}=0\,|\,V_{i}=v) f_{i-1}(v\,|\,0) \diff v- \int_0^\infty v\,\P(V_1=0\,|\,V_{0}=v) f_{0}(v) \diff v \\ 
&= \frac{\xi}{\theta_k}\left(\E\left[V_{i-1} {\rm e}^{-\theta_k V_{i-1}}\,|\,V_0=0\right]-\E\left[V {\rm e}^{-\theta_k V}\right]\right) \ .
\end{split} 
\]
Applying \eqref{eq:EVeV_convegence} in Lemma \ref{lemma:LST_convergence} we conclude that if $|\varphi_k^{(3)}(0)|<\infty$, then
$
\sum_{i=1}^\infty |a_i\,|\,<\infty.
$

{\it --~Second term.} Applying similar arguments, the second term term of \eqref{eq:cov_ki} is given by
\[
w_2\left(\E\left[V_0 Y_1g(V_{i}) (1-Y_{i+1})\right]-\E\left[V_0 Y_1\right]\E\left[g(V_{0})(1-Y_{1})\right]\right)=w_2\,\E\left[V_0 Y_1\right]b_i\ ,
\]
where 
\[
b_i:=\E\left[ g(V_{i})\mathbf{1}(V_{i+1}>0)\,|\,V_1=0\right]-\E\left[g(V_{0})\mathbf{1}(V_1>0)\right] \ .
\]
In order to apply Lemma \ref{lemma:LST_convergence} we first apply a Taylor expansion to $b_i$. First, conditioning on $V_{i-1}$ and applying \eqref{eq:P_idle_transient} yields
\[
\E\left[ g(V_{i})\mathbf{1}(V_{i+1}>0)|V_1=0\right] =\E\left[ \left(1-\frac{\xi}{\theta_k} {\rm e}^{-\theta_k V_{i}}\right)g(V_{i})|V_1=0\right] \ .
\]
We thus have that $b_i=c_i-d_i$, where
\begin{align*}
c_i&:=\E\left[ g(V_{i})\,|\,V_1=0\right]-\E\left[g(V_{0})\right] \ ,
\\
d_i&:=\E\left[\frac{\xi}{\theta_k} {\rm e}^{-\theta_k V_{i}}g(V_{i})\,|\,V_1=0\right]-\E\left[\frac{\xi}{\theta_k} {\rm e}^{-\theta_k V_{0}}g(V_{0})\right] \ .
\end{align*}
Applying the Taylor expansion \eqref{eq:g_Taylor} yields
\[
\begin{split}
\sum_{i=1}^\infty c_i &= \sum_{j=1}^\infty\frac{1}{j}\left(\frac{\xi}{\theta_0}\right)^j\sum_{i=1}^\infty\left(\E\left[ {\rm e}^{-\theta_0 j V_{i}}\,|\,V_1=0\right]-\E\left[ {\rm e}^{-\theta_0 j V_{0}}\right]\right) \,-\\
&\quad\:\sum_{j=1}^\infty\frac{1}{j}\left(\frac{\xi}{\theta_1}\right)^j\sum_{i=1}^\infty\left(\E\left[ {\rm e}^{-\theta_1 j V_{i}}\,|\,V_1=0\right]-\E\left[ {\rm e}^{-\theta_1 j V_{0}}\right]\right) \ ,
\end{split}
\]
We prove the finiteness of the first term in the right-hand side in the previous display; the other term follows analogously.
By Corollary \ref{C2}, and recalling $\xi<\theta_0$, 
\[\sum_{j=1}^\infty\frac{1}{j}\left(\frac{\xi}{\theta_0}\right)^j\sum_{i=1}^\infty\left(\E\left[ {\rm e}^{-\theta_0 j V_{i}}\,|\,V_1=0\right]-\E\left[ {\rm e}^{-\theta_0 j V_{0}}\right]\right)\leqslant \Xi \sum_{j=1}^\infty\frac{1}{j}\left(\frac{\xi}{\theta_0}\right)^j =-\Xi\log\left(1-\frac{\xi}{\theta_0}\right). \]
Hence by applying the triangle inequality we conclude that $\sum_{i=1}^\infty |c_i\,|\, <\infty$. Similarly, a Taylor expansion \eqref{eq:g_Taylor} yields
\[
\begin{split}
d_i &= \frac{\xi}{\theta_k}\frac{1}{w_1}\left(\E\left[ {\rm e}^{-\theta_k V_{i}}\,|\,V_1=0\right]-\E\left[ {\rm e}^{-\theta_k V_{0}}\right] \right) \\
&\quad +\frac{\xi}{\theta_k} \sum_{j=1}^\infty\frac{1}{j}\left(\frac{\xi}{\theta_0}\right)^j\left(\E\left[ {\rm e}^{-(\theta_0 j+\theta_k) V_{i}}\,|\,V_1=0\right]-\E_k\left[ {\rm e}^{-(\theta_0 j+\theta_k) V_{0}}\right]\right) \\
&\quad -\frac{\xi}{\theta_k}\sum_{j=1}^\infty\frac{1}{j}\left(\frac{\xi}{\theta_1}\right)^j\left(\E\left[ {\rm e}^{-(\theta_1 j+\theta_k) V_{i}}\,|\,V_1=0\right]-\E\left[ {\rm e}^{-(\theta_1 j+\theta_k)V_{0}}\right]\right) \ .
\end{split}
\]
Again, by applying Lemma \ref{lemma:LST_convergence} and Corollary \ref{C2} we conclude that $\sum_{i=1}^\infty |d_i\,|\, <\infty$, and by the triangle inequality that
$
\sum_{i=1}^\infty |b_i\,|\,<\infty.$

{\it --~Third term.} By conditioning on $V_1>0$, the third term of \eqref{eq:cov_ki} equals
\[
w_2\left(\E\left[g(V_0)(1-Y_1)V_i Y_{i+1}\right] - \E\left[V_0 Y_1\right]\E\left[(1-Y_{1})g(V_{0})\right]\right)=w_2\,\E\left[g(V_0)(1-Y_1)\right]e_i\ ,
\]
where
\[
e_i:=\E\left[V_{i}\mathbf{1}(V_{i+1}=0)|V_1>0\right]-\E\left[V_0\mathbf{1}(V_1=0)\right]\ .
\]
Now observe that
\begin{align*} \bar e_i\::=&\,\E\left[V_{i}\mathbf{1}(V_{i+1}=0)\,|\,V_1=0\right] - \E\left[V_{i}\mathbf{1}(V_{i+1}=0)\,|\,V_1>0\right]\\
=&\,\E\left[V_{i}\mathbf{1}(V_{i+1}=0)\,|\,V_1=0\right] - \frac{\E\left[V_{i}\mathbf{1}(V_{i+1}=0)\right]-\E\left[V_{i}\mathbf{1}(V_{i+1}=0)\,|\,V_1=0\right]{\P(V_1=0)}}{\P(V_1>0)}\\
=&\, \frac{\E\left[V_{i}\mathbf{1}(V_{i+1}=0)\,|\,V_1=0\right] - \E\left[V_{i}\mathbf{1}(V_{i+1}=0)\right]}{\P(V_1>0)} =\frac{a_i}{\varphi'(0)}. 
\end{align*} From $e_i=a_i-\bar e_i=a_i(1-1/\varphi'(0))$ and the first term, it follows that $\sum_{i=1}^\infty |e_i\,|\,<\infty.$

{\it --~Fourth term.} The last term of \eqref{eq:cov_ki} can be expressed as
\[
\E\left[ g(V_{0})(1-Y_{1})g(V_{i})(1-Y_{i+1})\right]-\E^2\left[g(V_0) (1-Y_{1})\right]=\E\left[g(V_0)(1-Y_1)\right]f_i\ ,
\]
where
\[
f_i:=\E\left[g(V_{i})\mathbf{1}(V_{i+1}>0)\,|\,V_1>0\right]-\E\left[g(V_0)\mathbf{1}(V_1>0)\right]\ .
\]
Then it is a matter of straightforwardly combining ideas from the second and third term to prove that $\sum_{i=1}^\infty |f_i\,|\,<\infty.$

Now that we have bounds on all four terms, by applying the triangle inequality once more we conclude that there exists a constant $0<\kappa <\infty$ such that
\[
\sum_{i=1}^{\infty}|c_{ki}| \leqslant \kappa \left(\sum_{i=1}^{\infty}|a_{i}|+\sum_{i=1}^{\infty}|b_{i}|+\sum_{i=1}^{\infty}|e_{i}|+\sum_{i=1}^{\infty}|f_{i}|\right) <\infty\ ,
\]
and we thus conclude that $\sigma_k^2<\infty$. 
\end{proof}

\bibliographystyle{abbrv}
\small{
}

\end{document}